\documentclass[11pt, reqno, twoside, makeidx]{amsart}
\usepackage[margin=2.7cm, marginpar=1cm]{geometry}

\usepackage{wrapfig}
\usepackage{marginnote}
\usepackage{hyperref}
\usepackage[english]{babel}
\usepackage[utf8]{inputenc}
\usepackage[T1]{fontenc}
\usepackage{amsmath,amsthm,amssymb,amsfonts}
\usepackage{relsize}
\numberwithin{equation}{section}
\usepackage{enumerate}
\usepackage{faktor}
\usepackage{dsfont}
\usepackage{scalerel,stackengine}
\usepackage{textcomp}
\usepackage{graphicx}
\usepackage{extarrows}
\usepackage{epigraph}
\usepackage{graphicx}
\usepackage{subcaption}
\usepackage{sidecap}
\usepackage{indentfirst}
\usepackage{tikz}
\usepackage{tikz-cd}
\usepackage{tkz-euclide}
\usetikzlibrary{shapes.misc,arrows,matrix,patterns,decorations.markings,positioning}
\tikzset{cross/.style={cross out, draw=black, minimum size=2*(#1-\pgflinewidth), inner sep=0pt, outer sep=0pt},
cross/.default={4.5pt}}
\usepackage{pgfplots}
\usepackage{caption}
\usepackage{float}
\usepackage{color}
\usepackage{epstopdf}
\usepackage{epsfig}
\usepackage{stmaryrd}
\usepackage{color}
\usepackage{geometry}
\usepackage{multicol}
\usepackage{verbatim}
\usepackage{wrapfig}
\usepackage{float}

\DeclareMathOperator{\Ker}{Ker }
\DeclareMathOperator{\Imm}{Im }
\DeclareMathOperator{\tb}{tb}
\DeclareMathOperator{\rot}{rot}

\DeclareMathOperator{\lk}{\ell k}

\DeclareMathOperator{\Spin}{Spin}

\DeclareMathOperator{\Id}{Id}

\DeclareMathOperator{\xist}{\xi_{std}}
\renewcommand{\geq}{\geqslant}
\renewcommand{\leq}{\leqslant} 
\renewcommand{\epsilon}{\varepsilon}

\newcommand{\Z}{\mathbb{Z}}
\newcommand{\Q}{\mathbb{Q}}
\newcommand{\F}{\mathbb{F}}

\newcommand{\C}{\mathbb C}
\newcommand{\J}{\mathcal J}

\newcommand{\s}{\mathfrak{s}}

\newtheorem{teo}{Theorem}[section]
\newtheorem*{teo*}{Theorem}
\newtheorem{lemma}[teo]{Lemma}
\newtheorem{prop}[teo]{Proposition}
\newtheorem*{prop*}{Proposition}

\newtheorem{cor}[teo]{Corollary}

\usepackage{xpatch}
\makeatletter
\xpatchcmd{\@thm}{\thm@headpunct{.}}{\thm@headpunct{}}{}{}
\makeatother

\pgfplotsset{compat=1.18}
\begin{document}
\title[Mazur manifolds and symplectic structures]{Mazur manifolds and symplectic structures}
\author{Alberto Cavallo}
\address{HUN-REN Alfr\'ed R\'enyi Institute of Mathematics, Budapest 1053, Hungary}
\email{acavallo@impan.pl}


\begin{abstract}
 We use the Heegaard Floer homology cobordism maps to obstruct the existence of a symplectic structure on the Akbulut-Kirby Mazur manifolds whose boundary is a Brieskorn sphere $Y$ among $\Sigma(2,3,13),$ $\Sigma(2,5,7)$ and $\Sigma(3,4,5)$. Furthermore, we describe how they imply the existence of exotic pairs of simply connected 4-manifolds, with definite intersection form, whose boundary is a Brieskorn sphere bounding a Mazur manifold.
\end{abstract}

\maketitle

\thispagestyle{empty}

\section{Introduction}
In \cite{ACM} we provide evidence for a positive answer to the generalized version of a conjecture of Gompf; namely, that no Brieskorn sphere $Y=\Sigma(a_1,...,a_n)$ arises as the boundary of a rational homology ball symplectic filling. What we show in \cite[Theorems 1.4 and 1.5]{ACM} is that the contact structure $\xi_\text{can}$, usually called the Milnor fillable contact structure on $Y$, admits no such filling except for some specific cases: the Brieskorn spheres $\Sigma(3,4,5),$ $\Sigma(2,5,7)$ and $\Sigma(2,3,6k+1)$ when $k\geq1$. For the latter manifolds it is still unknown whether rational homology ball symplectic fillings may exist or not. Note that it follows from our results, see \cite[Proposition 1.7]{ACM}, that these are the only canonically oriented Brieskorn spheres carrying exactly two tight structures, up to isotopy, and with vanishing correction term.

The starting obstruction that we may look at is whether $Y$ bounds a rational homology ball in the first place. It is known from the literature that $\Sigma(3,4,5),$ $\Sigma(2,5,7),$ $\Sigma(2,3,13)$ and $\Sigma(2,3,25)$ bound contractible Mazur manifolds, see \cite{AK,Fickle}, while $\Sigma(2,3,7)$ and $\Sigma(2,3,19)$ both bound rational homology balls, but not an integral one because they have non-vanishing $\overline\mu$-invariant, see \cite{FS,AL}. 

\begin{figure}[h] 
 \begin{tikzpicture}[scale=0.75]
    \tkzDefPoints{0/0/A, 1.5/1/B, 1.5/-1/D, 3/-1/E, 1.5/0/C}  
    \tkzDrawSegment(A,B)\tkzDrawSegment(A,D)\tkzDrawSegment(E,D)\tkzDrawSegment(A,C)
    \tkzDrawPoints[fill,black,size=5](A,B,C,D,E)
     \tkzLabelPoint[above left](A){$-1$} \tkzLabelPoint[above left](B){$-2$}\tkzLabelPoint[right](C){$-3$}
     \tkzLabelPoint[below left](D){$-7$}\tkzLabelPoint[below left](E){$-2$}
\end{tikzpicture}\hspace{1cm} 
       \begin{tikzpicture}[scale=0.75]
    \tkzDefPoints{0/0/A, 1.5/1/B, 1.5/-1/D, 3/-1/E, 1.5/0/C}  
    \tkzDrawSegment(A,B)\tkzDrawSegment(A,D)\tkzDrawSegment(E,D)\tkzDrawSegment(B,A)\tkzDrawSegment(A,C)
    \tkzDrawPoints[fill,black,size=5](A,B,C,D,E)
     \tkzLabelPoint[above left](A){$-1$} 
     \tkzLabelPoint[above left](B){$-2$}\tkzLabelPoint[right](C){$-5$}\tkzLabelPoint[below left](D){$-4$}\tkzLabelPoint[below left](E){$-2$}
       \end{tikzpicture}\hspace{1cm}
      \begin{tikzpicture}[scale=0.75]
    \tkzDefPoints{0/0/A, 1.5/1/B, 1.5/-1/D, 3/-1/E, 1.5/0/C}  
    \tkzDrawSegment(A,B)\tkzDrawSegment(A,D)\tkzDrawSegment(E,D)\tkzDrawSegment(B,A)\tkzDrawSegment(A,C)
    \tkzDrawPoints[fill,black,size=5](A,B,C,D,E)
     \tkzLabelPoint[above left](A){$-1$} 
     \tkzLabelPoint[above left](B){$-3$}\tkzLabelPoint[right](C){$-4$}
     \tkzLabelPoint[below left](D){$-3$}\tkzLabelPoint[below left](E){$-2$}
       \end{tikzpicture}
     \caption{\smaller[1]{The standard graph of $\Sigma(2,3,13)$ (left), $\Sigma(2,5,7)$ (middle) and $\Sigma(3,4,5)$ (right).}}
     \label{Graph}
\end{figure}
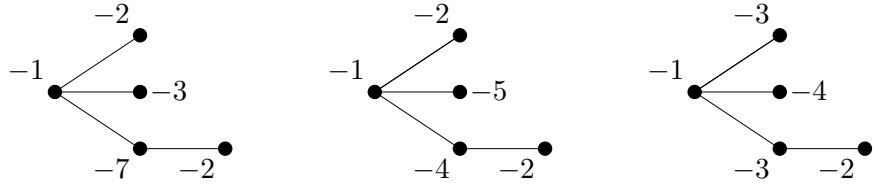 

It is a result of Mark and Tosun \cite[Theorem 1.8]{MT2} that the Akbulut-Kirby Mazur manifold $\pm W^-(2)$, whose boundary is $\pm\Sigma(2,3,13)$, does not carry a Stein structure. In this paper we focus on $\Sigma(2,3,13),$ $\Sigma(2,5,7)$ and $\Sigma(3,4,5)$, whose plumbing graph is drawn in Figure \ref{Graph}, 
and prove the following stronger result using Heegaard Floer homology.
\begin{teo}
 \label{teo:main}
 The connected sum of the Mazur manifold $W^-(m)$ in Figure \ref{Mazur} with any closed $4$-manifold is not a symplectic filling for $m=2,3,4$.   
\end{teo}
Note that our argument does not imply that $\Sigma(2,3,13),$ $\Sigma(2,5,7)$ and $\Sigma(3,4,5)$ cannot bound rational homology ball symplectic fillings which do not admit a summand diffeomorphic to $W^-(m)$.

A \emph{Mazur manifold}\footnote{Usually only the ones with exactly one 1-handle and one 2-handle are considered Mazur manifolds; the 4-manifolds appearing here are of this kind. We consider the given definition as the natural extension to manifolds with a richer handle decomposition.} is a smooth 4-manifold $X$ which can be decomposed as the union of only one 0-handle and an equal number of 1- and 2-handles, so that its double $D(X):=X\cup_{\partial X}-X$ is the standard 4-sphere. In particular, the manifold $X$ is contractible and has non-empty boundary. Some of the first examples of Mazur manifolds were given by Akbulut and Kirby in \cite{AK}, see Figure \ref{Mazur}; this family includes the manifolds $W^-(m)$ with $m=2,3,4$ which we are going to study in this paper.

\begin{figure}[h]
 \centering
 \vspace{0cm}
  \def\svgwidth{0.48\textwidth}
\begingroup%
  \makeatletter%
  \providecommand\color[2][]{%
    \errmessage{(Inkscape) Color is used for the text in Inkscape, but the package 'color.sty' is not loaded}%
    \renewcommand\color[2][]{}%
  }%
  \providecommand\transparent[1]{%
    \errmessage{(Inkscape) Transparency is used (non-zero) for the text in Inkscape, but the package 'transparent.sty' is not loaded}%
    \renewcommand\transparent[1]{}%
  }%
  \providecommand\rotatebox[2]{#2}%
  \newcommand*\fsize{\dimexpr\f@size pt\relax}%
  \newcommand*\lineheight[1]{\fontsize{\fsize}{#1\fsize}\selectfont}%
  \ifx\svgwidth\undefined%
    \setlength{\unitlength}{1135.15647107bp}%
    \ifx\svgscale\undefined%
      \relax%
    \else%
      \setlength{\unitlength}{\unitlength * \real{\svgscale}}%
    \fi%
  \else%
    \setlength{\unitlength}{\svgwidth}%
  \fi%
  \global\let\svgwidth\undefined%
  \global\let\svgscale\undefined%
  \makeatother%
  \begin{picture}(1,0.65215127)%
    \lineheight{1}%
    \setlength\tabcolsep{0pt}%
    \put(0,0){\includegraphics[width=\unitlength,page=1]{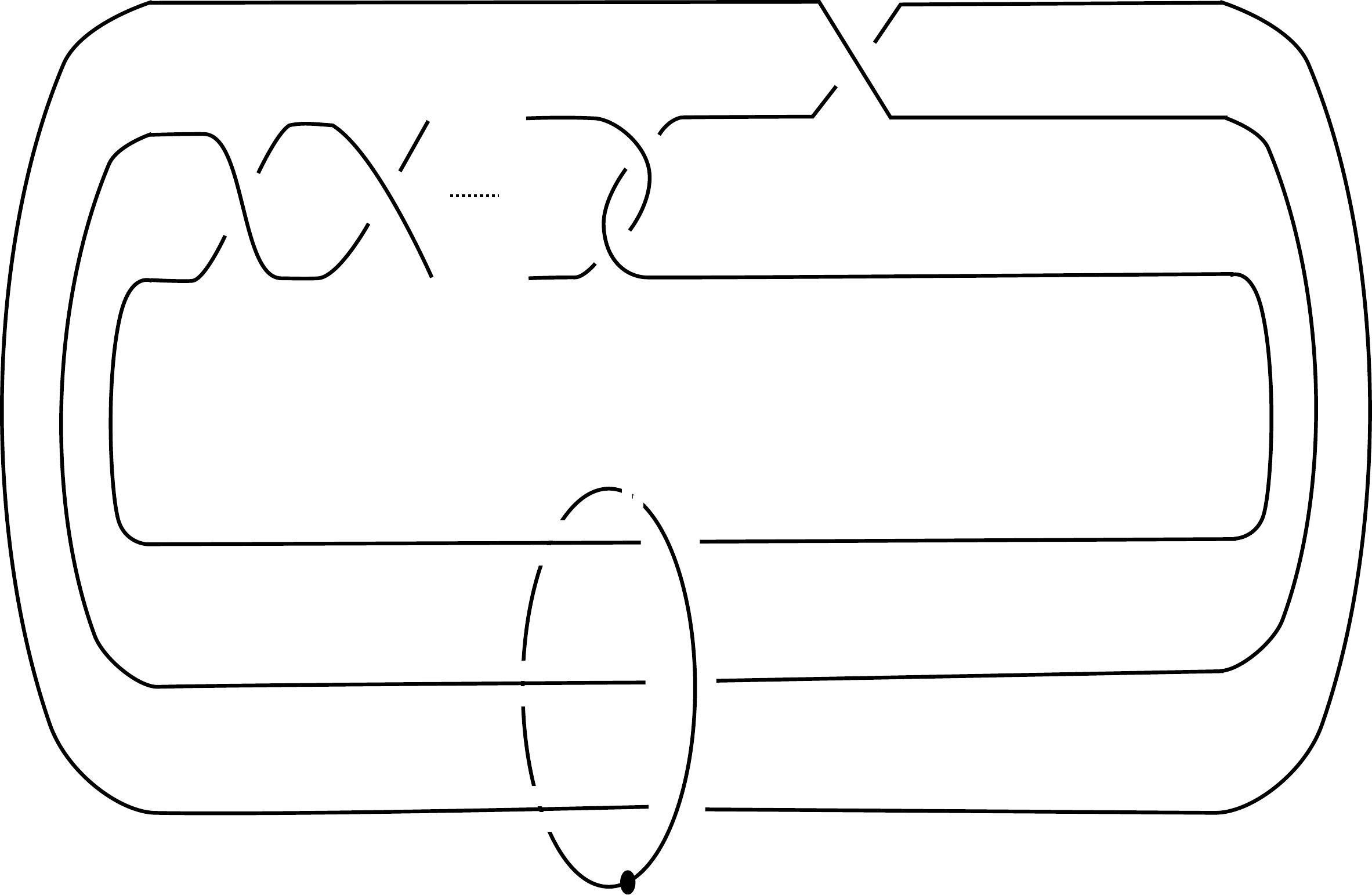}}%
    \put(0.79801256,0.28008267){\color[rgb]{0,0,0}\makebox(0,0)[lt]{\lineheight{1.25}\smash{\begin{tabular}[t]{l}$k$\end{tabular}}}}%
    \put(0.46878121,0.3426178){\color[rgb]{0,0,0}\makebox(0,0)[lt]{\lineheight{1.25}\smash{\begin{tabular}[t]{l}$\gamma_-$\end{tabular}}}}%
    \put(0.11080241,0.3890421){\color[rgb]{0,0,0}\makebox(0,0)[lt]{\lineheight{1.25}\smash{\begin{tabular}[t]{l}$\ell\text{ full twists}$\end{tabular}}}}%
    \put(0,0){\includegraphics[width=\unitlength,page=2]{Mazur-.pdf}}%
  \end{picture}%
\endgroup%

  \vspace{0cm}   
  \vspace{0cm}
  \def\svgwidth{0.48\textwidth}
\begingroup%
  \makeatletter%
  \providecommand\color[2][]{%
    \errmessage{(Inkscape) Color is used for the text in Inkscape, but the package 'color.sty' is not loaded}%
    \renewcommand\color[2][]{}%
  }%
  \providecommand\transparent[1]{%
    \errmessage{(Inkscape) Transparency is used (non-zero) for the text in Inkscape, but the package 'transparent.sty' is not loaded}%
    \renewcommand\transparent[1]{}%
  }%
  \providecommand\rotatebox[2]{#2}%
  \newcommand*\fsize{\dimexpr\f@size pt\relax}%
  \newcommand*\lineheight[1]{\fontsize{\fsize}{#1\fsize}\selectfont}%
  \ifx\svgwidth\undefined%
    \setlength{\unitlength}{1100.6586222bp}%
    \ifx\svgscale\undefined%
      \relax%
    \else%
      \setlength{\unitlength}{\unitlength * \real{\svgscale}}%
    \fi%
  \else%
    \setlength{\unitlength}{\svgwidth}%
  \fi%
  \global\let\svgwidth\undefined%
  \global\let\svgscale\undefined%
  \makeatother%
  \begin{picture}(1,0.67259214)%
    \lineheight{1}%
    \setlength\tabcolsep{0pt}%
    \put(0,0){\includegraphics[width=\unitlength,page=1]{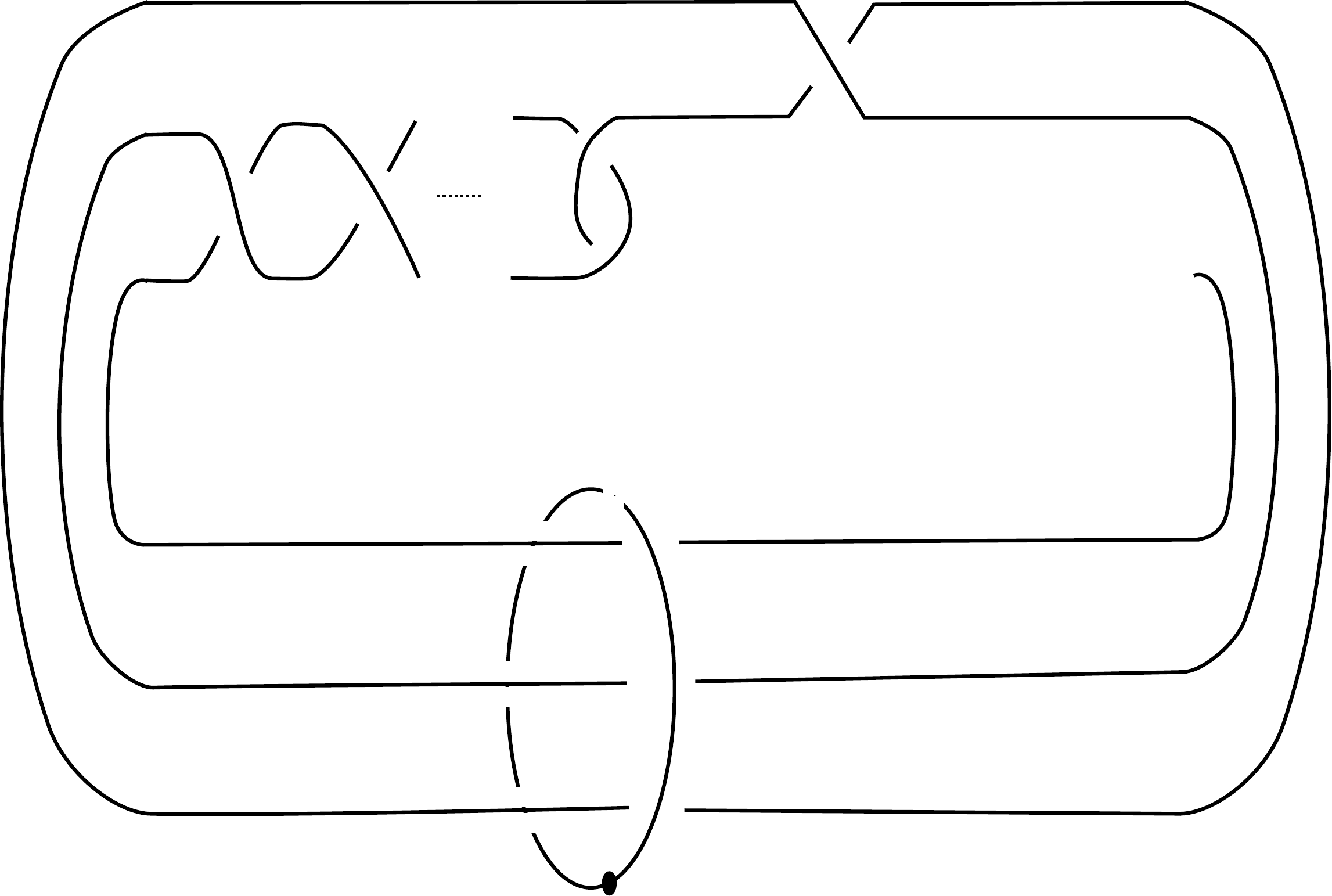}}%
    \put(0.79800025,0.28887197){\color[rgb]{0,0,0}\makebox(0,0)[lt]{\lineheight{1.25}\smash{\begin{tabular}[t]{l}$k$\end{tabular}}}}%
    \put(0.46878252,0.35336951){\color[rgb]{0,0,0}\makebox(0,0)[lt]{\lineheight{1.25}\smash{\begin{tabular}[t]{l}$\gamma_+$\end{tabular}}}}%
    \put(0.10905499,0.40582059){\color[rgb]{0,0,0}\makebox(0,0)[lt]{\lineheight{1.25}\smash{\begin{tabular}[t]{l}$\ell\text{ full twists}$\end{tabular}}}}%
    \put(0,0){\includegraphics[width=\unitlength,page=2]{Mazur+.pdf}}%
  \end{picture}%
\endgroup%

  \vspace{0cm}
 \caption{\smaller[1]{The Mazur manifold $W^\pm(\ell,k)$, and the knot $\gamma_\pm$ on its boundary. Note that the diffeomorphism type of $W^-(\ell,k)$ only depends on $\ell+k$.}}
 \label{Mazur}
 \end{figure}

We recall the following result which is deduced from \cite[Proposition 1]{AK}. We write that $X_1$ is (orientation-preserving) diffeomorphic to $X_2$ by $X_1\cong X_2$.
\begin{teo}[Akbulut-Kirby]
 \label{teo:AK}
 The Mazur manifold $W^\pm(\ell,k)$ satisfies the following properties:
 \begin{enumerate}
     \item $W^\pm(\ell,k)\cong W^\pm(\ell+1,k-1)$ for every $\ell,k\in\Z$;
     \item $W^-(\ell,k)\cong-W^+(-\ell+2,-k+1)$ and the diffeomorphism identifies $\gamma_-$ with $\gamma_+$ for every $\ell,k\in\Z$;
     \item $\gamma_-$ is smoothly slice in $W^-(\ell,k)$ if $(\ell,k)=(2,k)$ for every $k\in\Z$.
 \end{enumerate}
\end{teo}
From Property 1) above we can immediately deduce that the diffeomorphism type
of $W^\pm(\ell,k)$ only depends on $\ell+k$; hence, from now on we may denote the manifolds in Figure \ref{Mazur} by $W^\pm(m)$ and their boundaries by $Y^\pm(m)$ where $m\in\Z$. In addition, from Property 2) we see that $W^-(m)\cong-W^+(-m+3)$.
In a similar way, combining Properties 2) and 3) yields that $\gamma_+$ is smoothly slice in $W^+(\ell,k)$ if $(\ell,k)=(0,h)$ where $h\in\Z$.

The manifold $W^-(0)$ was the first known example of \emph{cork}, a contractible 4-manifold $X$ whose boundary has an involution which does not extend to $X$, and it has been extensively studied in literature, see for example \cite{AD,Cork,DHM}. Moreover, it is proved in \cite[Theorem 2]{AK} that the Brieskorn spheres we are concerned with here appear as the boundary\footnote{In \cite{AK} Akbulut and Kirby do not specify the orientation of $W^\pm(m)$. By checking their proofs explicitly, one can see that the canonical orientation, induced by the Milnor fiber, is opposite to the one in \cite[Theorem 2]{AK} while it agrees with the one in \cite{Akbulut}.} of the 4-manifolds in Figure \ref{Mazur}:
\[Y^-(m)\cong\left\{\begin{aligned}\Sigma(2,3,13)&\text{ if }m=2\:;\\
    \Sigma(2,5,7)&\text{ if }m=3\:; \\
    \Sigma(3,4,5)&\text{ if }m=4\;.\end{aligned}\right.\]

It is very easy to observe that almost all these manifolds carry a Stein structure with at least one orientation; such structures are shown in Figure \ref{Stein}. The ones corresponding to the three cases above are among the exceptions. 

\begin{figure}[h]
 \centering
 \vspace{0cm}
  \def\svgwidth{0.48\textwidth}
\begingroup%
  \makeatletter%
  \providecommand\color[2][]{%
    \errmessage{(Inkscape) Color is used for the text in Inkscape, but the package 'color.sty' is not loaded}%
    \renewcommand\color[2][]{}%
  }%
  \providecommand\transparent[1]{%
    \errmessage{(Inkscape) Transparency is used (non-zero) for the text in Inkscape, but the package 'transparent.sty' is not loaded}%
    \renewcommand\transparent[1]{}%
  }%
  \providecommand\rotatebox[2]{#2}%
  \newcommand*\fsize{\dimexpr\f@size pt\relax}%
  \newcommand*\lineheight[1]{\fontsize{\fsize}{#1\fsize}\selectfont}%
  \ifx\svgwidth\undefined%
    \setlength{\unitlength}{1600.46198327bp}%
    \ifx\svgscale\undefined%
      \relax%
    \else%
      \setlength{\unitlength}{\unitlength * \real{\svgscale}}%
    \fi%
  \else%
    \setlength{\unitlength}{\svgwidth}%
  \fi%
  \global\let\svgwidth\undefined%
  \global\let\svgscale\undefined%
  \makeatother%
  \begin{picture}(1,0.27927833)%
    \lineheight{1}%
    \setlength\tabcolsep{0pt}%
    \put(0,0){\includegraphics[width=\unitlength,page=1]{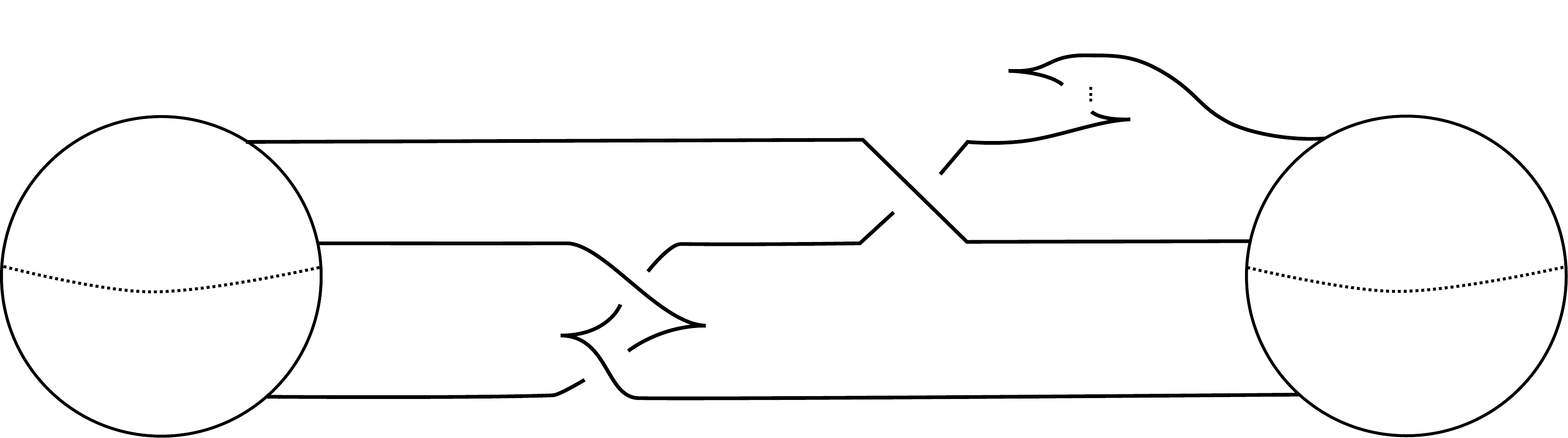}}%
    \put(0.62759032,0.26642254){\color[rgb]{0,0,0}\makebox(0,0)[lt]{\lineheight{1.25}\smash{\begin{tabular}[t]{l}$s\text{ times}$\end{tabular}}}}%
  \end{picture}%
\endgroup%

  \vspace{0cm}   
  \vspace{0cm}
  \def\svgwidth{0.48\textwidth}
\begingroup%
  \makeatletter%
  \providecommand\color[2][]{%
    \errmessage{(Inkscape) Color is used for the text in Inkscape, but the package 'color.sty' is not loaded}%
    \renewcommand\color[2][]{}%
  }%
  \providecommand\transparent[1]{%
    \errmessage{(Inkscape) Transparency is used (non-zero) for the text in Inkscape, but the package 'transparent.sty' is not loaded}%
    \renewcommand\transparent[1]{}%
  }%
  \providecommand\rotatebox[2]{#2}%
  \newcommand*\fsize{\dimexpr\f@size pt\relax}%
  \newcommand*\lineheight[1]{\fontsize{\fsize}{#1\fsize}\selectfont}%
  \ifx\svgwidth\undefined%
    \setlength{\unitlength}{1600.46198327bp}%
    \ifx\svgscale\undefined%
      \relax%
    \else%
      \setlength{\unitlength}{\unitlength * \real{\svgscale}}%
    \fi%
  \else%
    \setlength{\unitlength}{\svgwidth}%
  \fi%
  \global\let\svgwidth\undefined%
  \global\let\svgscale\undefined%
  \makeatother%
  \begin{picture}(1,0.27927833)%
    \lineheight{1}%
    \setlength\tabcolsep{0pt}%
    \put(0,0){\includegraphics[width=\unitlength,page=1]{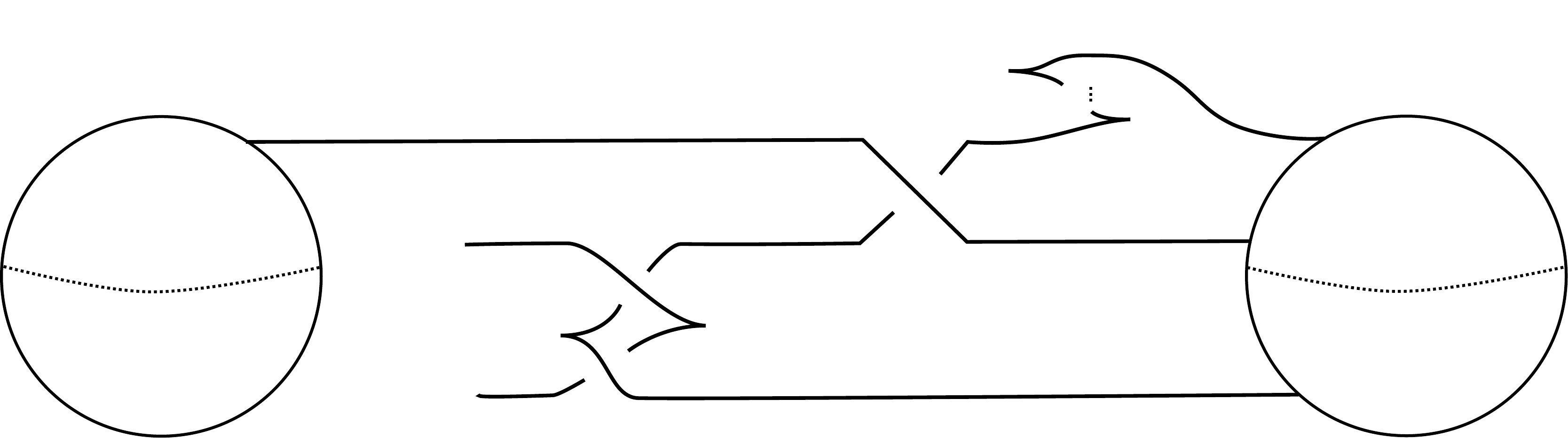}}%
    \put(0.62759032,0.26642254){\color[rgb]{0,0,0}\makebox(0,0)[lt]{\lineheight{1.25}\smash{\begin{tabular}[t]{l}$s\text{ times}$\end{tabular}}}}%
    \put(0,0){\includegraphics[width=\unitlength,page=2]{Stein+.pdf}}%
  \end{picture}%
\endgroup%

  \vspace{0cm}
 \caption{\smaller[1]{Stein structures on $W^-(0,1-s)$ (left) and $W^+(1,-4-s)$ (right). It is easy to check that the Legendrian knots where we attach the Stein 2-handles have $\tb$-number equal to $2-s$ and $-3-s$, where $s$ is the number of stabilizations.}}
 \label{Stein}
 \end{figure}

Ozsv\'ath and Szab\'o \cite{OSz-negative} defined the cobordism map $F^\circ_{W,\mathfrak u}:HF^\circ(S^3)\rightarrow HF^\circ(Y:=\partial W,\mathfrak u\lvert_Y)$ induced by any 4-manifold $W$, equipped with a $\Spin^c$-structure $\mathfrak u$; in particular, when we take $W=W^\pm(m)$ then we do not need any additional data, as integral homology balls admit a unique $\Spin^c$-structure. We recall that cobordism maps have a well-defined degree-shift, which is shown to be vanishing when $b_1(W)=b_2(W)=0$.  

We know from \cite{OSz-negative} that if $b_1(W)=0$ and $W$ is negative-definite then $F^-_{W,\mathfrak u}(1)$ is a non-torsion element of the $\F[U]$-module $HF^-(Y,\mathfrak u\lvert_Y)$, where $\F$ is the field with two elements. Hence, we have that  $F^-_{W^\pm(m)}(1)\in HF^-(Y^\pm(m))$ is non-torsion for each $m$, and is homogeneous with Maslov grading zero. When $m=2,3,4$ our manifolds are canonically oriented Brieskorn spheres, thus they are presented by an almost-rational graph, see \cite{Nemethi0} for the definition. It follows that we can produce a canonical basis $\mathcal B=\{[V_1],...,[V_t]\}$ of $HF^-(Y^-(m))/U\cdot HF^-(Y^-(m))$ for $m=2,3,4$ in terms of the full paths of their graph, we refer to \cite{OSz-fullpath,CM-negative} for more details. 

In Section \ref{section:three} we determine the coordinates of $F^-_{W^\pm(m)}(1)$ with respect to $\mathcal B$ when $m=3,4$, while we give restrictions in the case that $m=2$. We achieve this by computing the $\tau$-invariants of the knot $\gamma^-\subset Y^-(m)$ in Figure \ref{Mazur}, which is smoothly slice in $W^-(2,m-2)$ by Theorem \ref{teo:AK}. In the paper we use different versions of the $\tau$-invariant of a knot $K\subset Y$: the first is $\tau_\xi(K)$ introduced by Hedden in \cite{Hedden} when $Y$ carries a contact structure $\xi$ with $\widehat c(\xi)\neq0$, while the second one is $\tau_\theta(K)$ introduced by the author and Alfieri in \cite{AC}, see Section \ref{section:two} for the definition. Using \cite[Theorem 1.3]{AC} we can easily compute $\tau_\xi(K)$ when the knot $K$ is the boundary of a symplectic surface in a Stein domain; this means that under the right assumptions we can write an explicit formula for the invariant.

\begin{prop}
 \label{teo:AC}
 Consider the knot $K\subset Y=S^3_\Lambda(L)$ with $\det(\Lambda)\neq0$ such that $K\cup L$ has a Legendrian realization $\mathcal K\cup\mathcal L$ in $(S^3,\xist)$ where $\Lambda_{ii}=\tb_{\xist}(\mathcal L_i)-1$ for $i=1,...,|L|$, and $\mathcal K$ bounds a Lagrangian surface in $D^4$. Then we have the following formula: \[2\tau_\xi(\pm K)-1=\tb_{\xist}(\mathcal K)-L^T\Lambda^{-1}L\pm L^T\Lambda^{-1}\mathbf V\:,\] where \[L=\big(\lk(K,L_1),...,\lk(K,L_{|L|})\big)\:,\hspace{0.5cm}\mathbf V=\big(\rot_{\xist}(\mathcal L_1),...,\rot_{\xist}(\mathcal L_{|L|})\big)\] and $\xi$ is obtained by Legendrian surgery on $\mathcal L$. In particular, this holds when $K$ is the dual knot of a component of $L$.  
\end{prop}
\begin{proof}
 Since $\mathcal K$ bounds a Lagrangian surface $\Sigma$ in $D^4$, equipped with its unique Stein structure, we have that the knot $\pm K$ admits a transverse realization which bounds a symplectic curve. From a result of Chantraine \cite{Baptiste} we obtain that \begin{equation}\tau(\pm K)=g_4(D^4,K)=g(\Sigma)=\dfrac{\tb_{\xist}(\mathcal K)+1}{2}\:,\label{eq:Baptiste}\end{equation} where here $\tau(K)$ stands for the classical $\tau$-invariant in $S^3$. We now observe that $\mathcal K$ is Legendrian in the complement of $\mathcal L$; hence, the surface $\Sigma$ can be perturbed to be symplectic away from the Stein 2-handle attachments. Thus the perturbed surfaces remain symplectic also in the resulting Stein filling $(X,J)$ of $Y$, and we can apply \cite[Theorem 1.3]{AC} to get \[2\tau_\xi(\pm K)-1=2g(\Sigma)-1-[\Sigma]\cdot[\Sigma]\pm c_1(J)[\Sigma]\:.\] The Stein filling $(X,J)$ is such that $X$ is obtained by attaching 2-handles on $L$ with framings given by $\Lambda$, while the induced $\Spin^c$-structure is determined by the Chern class $c_1(J)=\mathbf V\in H^2(X;\Z)$, because $X$ is simply connected; we then conclude by observing that $\tb_{\xist}(\mathcal K)=2g(\Sigma)-1$ from Equation \eqref{eq:Baptiste}, while by linear algebra $[\Sigma]\cdot[\Sigma]=L^T\Lambda^{-1}L$ and $c_1(J)[\Sigma]=\mathbf V^T\Lambda^{-1}L=L^T\Lambda^{-1}\mathbf V$. 
\end{proof}

Note that in general $\tau_\xi(-K)=\tau_{\overline\xi}(K)$ where $\overline\xi$ is the conjugate of $\xi$, see \cite[Theorem 3.2]{CM-negative}. Note that some authors, see \cite{HR} for an example, use a different convention for the Alexander grading, interchanging the value of $\tau_\xi(K)$ and $\tau_\xi(-K)$; this is going to be irrelevant in this paper as the orientation of $K$ is not meaningful for our obstructions.

We apply our results to produce examples of exotic pairs of simply connected 4-manifolds, with definite intersection form, whose boundary is a Brieskorn sphere $Y$ bounding a Mazur manifold; these include $Y^-(m)$ for $m=2,3,4$. Our result is more general: let $Y$ be any 3-manifold which has a simply connected negative-definite Stein filling $X$ with $b_2(X)>0$ and bounds a Mazur manifold $W$, denote by $\widehat X$ any 4-manifold obtained by gluing $-W$ to $X$. 

\begin{teo}
 \label{teo:exotic}
 Consider $X_1:=X,$ $X_2:=\widehat X\#W$ and $X_3:=\#m\overline{\C P}^2\#W$, where $m=b_2(X)\geq1$, as above. The simply connected $4$-manifolds $X_1$ and $X_i$ with $i=2,3$ have definite-intersection form, and are homeomorphic but not diffeomorphic. 
\end{teo}

The only known example of such an exotic pair was found by Akbulut in \cite{Akbulut} with different techniques. We do not claim that the 4-manifolds $X_2$ and $X_3$ above are not diffeomorphic. If this were the case, then applying this result to $Y^-(m)$ would give the first examples of exotic simply connected closed 4-manifolds with definite intersection form. Nonetheless, we can state the following corollary.

\begin{cor}
 \label{cor:exotic}
 There exist two smooth embeddings $f$ and $g$ of a Brieskorn sphere $Y$, in a simply connected closed $4$-manifold $M$ with definite intersection form, such that $f(Y)$ and $g(Y)$ are topologically isotopic but no diffeomorphism of $M$ maps $f(Y)$ to $g(Y)$. Furthermore, the manifold $Y$ may bound a Mazur manifold $W\subset M$.
\end{cor}

\subsection*{Acknowledgments} {\smaller[1] I would like to thank the Matematiska institutionen at Uppsala universitet for their friendly hospitality. The author has been partially supported by the HORIZON-ERC-2023-ADG 101141468 KnotSurf4d project. }

\section{Tau-invariants of the knot \texorpdfstring{$\gamma_-$}{gamma-}}
\label{section:two}
Let us consider the Milnor fillable contact structure $\xi$ on $Y^-(2,m-2)\cong Y^-(m)$ where $m=3,4$, and then $Y^-(m)$ is either $\Sigma(2,5,7)$ or $\Sigma(3,4,5)$. We want to study the invariant $\tau_\xi(\gamma_-)$ defined by Hedden in \cite{Hedden}; more specifically, here we are going to consider the 2-element set $\mathcal T(m):=\{\tau_\xi(\gamma_-),\tau_{\overline\xi}(\gamma_-)\}$ where $\overline\xi$ is the contact structure obtained from $\xi$ by conjugation.

In general Hedden's invariant requires the knot to be oriented, but reversing the orientation acts in the same way as conjugating the structure: in other words, we have that $\tau_\eta(-K)=\tau_{\overline\eta}(K)$ for every knot $K$ in a contact 3-manifold $(Y,\eta)$ for which the $\tau$-invariant can be defined and $\widehat c(\eta)\neq0$. This implies that $\mathcal T(m)$ is actually an invariant of $\gamma_-$ as an unoriented knot, and thus in the remaining of the paper we do not specify any orientation for $\gamma_-$.

\subsection{Definition}
Let $Y$ be a rational homology 3-sphere.
For completeness we recall the definition of the invariant $\tau_\eta(K)$ for any knot $K\subset Y$ and $\eta$ is a contact structure on $Y$ with non-vanishing contact invariant, see \cite{Hedden}. Suppose that $\alpha\in\widehat{HF}(-Y)$ is a homogeneous non-zero homology class; then Hedden defined the invariant \[\tau_\alpha(K)=\min\{m\in\Q\:|\:\Imm i_*:H_*(\mathcal F^K_mCF(-Y))\rightarrow\widehat{HF}_*(-Y)\text{ contains }\alpha\}\] where $i:\mathcal F_m\widehat{CF}(-Y)\rightarrow\widehat{CF}(-Y)$ is the inclusion of the $m$-th level of the Alexander filtration, induced by $K$, in the Heegaard Floer chain complex of $-Y$. 

We then just set $\tau_\eta(K):=-\tau_{\widehat c(\eta)}(K)$ because $\widehat c(\eta)\in\widehat{HF}(-Y,\s_\eta)$, see \cite{OSz-contact}. Note that according to this definition, we have that for $(S^3,\xist)$ one has $\tau_{\xist}(K)=\tau(K)$ for every knot $K$.  

We now recall the definition of $\tau_\theta(K)$ from \cite[Subsection 2.2]{AC}. Say $K\subset Y$ is a knot, and suppose $Y$ is the boundary of a 4-manifold $X$ equipped with a $\Spin^c$-structure $\mathfrak u$. Then from Heegaard Floer theory \cite{OSz-negative} we construct the cobordism map $\widehat F_{X,\mathfrak u}:\widehat{HF}(S^3)\rightarrow\widehat{HF}(Y,\s)$ where $\s=\mathfrak u\lvert_Y$; hence, in the case that $\widehat F_{X,\mathfrak u}$ is non-trivial we have a distinguished non-zero element $\theta_\mathfrak u\in\widehat{HF}(Y,\s)$ defined as $\theta_\mathfrak u=\widehat F_{X,\mathfrak u}(1)$. We call $\tau_\theta(X,K,\mathfrak u)$ the invariant $\tau_{\theta_\mathfrak u}(K)$ associated to the homology class $\theta_\mathfrak u$. In other words, this is the minimal level of the Alexander filtration, induced by $K$ on $\widehat{HF}(Y,\s)$, that contains $\theta_\mathfrak u$.

Note that if $W$ is an integral homology 4-ball then for standard homological reasons there is a unique $\Spin^c$-structure on $W$, and thus we can write $\tau_\theta(W,K)$.

\subsection{Computation}
\label{subsection:computation}
We start to describe the argument: we construct a self-diffeomorphism\footnote{The Kirby moves have been performed by the KLO (Knot-Like Objects) software, version 0.978 alpha, available at \url{https://community.middlebury.edu/~mathanimations/klo/}.} $F$ of $Y^-(m)$ which maps a knot $K_m$ to $\gamma_-$, see Figures \ref{Diffeo1} and \ref{Diffeo2}, and we consider the contact structure obtained by the push-forward $F^*\xi$ of the Milnor fillable structure. We recall that we proved in \cite[Proposition 1.7]{ACM} that $\xi$ and its conjugate are the only symplectically fillable contact structures on both $\Sigma(2,5,7)$ and $\Sigma(3,4,5)$. 

It follows from \cite{OSz-contact} that the map induced by the contactomorphism $F^*:\widehat{HF}(-Y^-(m))\rightarrow\widehat{HF}(-Y^-(m))$ is an isomorphism and sends $\widehat c(\xi)$ to $\widehat c(F^*\xi)$. Moreover, the conjugation action induces an involution $\mathcal J$ on the Heegaard Floer groups, which has been studied by Ghiggini \cite{G-fillability}; he showed that $\mathcal J$ commutes with maps induced by diffeomorphisms and cobordisms, and that $\mathcal J\widehat c(\eta)=\widehat c(\overline\eta)$ for every contact structure $\eta$.
In our case this means that $F^*$ also sends $\widehat c(\overline\xi)$ to $\widehat c(\overline{F^*\xi})$; hence, the only possibility is that the structure $F^*\xi$ is either $\xi$ or $\overline\xi$, and we immediately see that $\mathcal T(m)$ is invariant under $F^*$. 

The presentation of $Y^-(m)$ in Figures \ref{Diffeo1} and \ref{Diffeo2} provides a Legendrian realization of $K_m$ in $(Y^-(m),\xi)$ which bounds a (homologically non-trivial) Lagrangian surface in the corresponding Stein domain $X(m)$, allowing us to compute $\mathcal T(m)$ by applying Proposition \ref{teo:AC}. We now give the details of the computation.

\begin{prop}
 \label{prop:tau}
 Suppose that $\gamma_-$ is the knot in $Y^-(2,m-2)$ for $m=3,4$ as in Figure \ref{Mazur}. Then $\mathcal T(m)=\{0,1\}$ when $m=3$, and $\mathcal T(m)=\{0,2\}$ when $m=4$.
\end{prop}
\begin{proof}
 The knot $K_m$ in the top of Figures \ref{Diffeo1} and \ref{Diffeo2} satisfies the assumptions of Proposition \ref{teo:AC}; in fact, the Stein filling $(X(m),J)$ is obtained by attaching Stein 2-handles in the complement of the standard Legendrian unknot in $(S^3,\xist)$, which bounds a Lagrangian disk in $D^4$ equipped with its (unique) Stein structure.
 
 The standard unknot has $\tb$-number equal to $-1$ and rotation number equal to zero, while the vectors $L$ and $\mathbf V$ appearing in the statement of Proposition \ref{teo:AC} are easily determined by the contact surgery presentations. 

\begin{figure}[t]
 \centering
  \def\svgwidth{0.6\textwidth}
\begingroup%
  \makeatletter%
  \providecommand\color[2][]{%
    \errmessage{(Inkscape) Color is used for the text in Inkscape, but the package 'color.sty' is not loaded}%
    \renewcommand\color[2][]{}%
  }%
  \providecommand\transparent[1]{%
    \errmessage{(Inkscape) Transparency is used (non-zero) for the text in Inkscape, but the package 'transparent.sty' is not loaded}%
    \renewcommand\transparent[1]{}%
  }%
  \providecommand\rotatebox[2]{#2}%
  \newcommand*\fsize{\dimexpr\f@size pt\relax}%
  \newcommand*\lineheight[1]{\fontsize{\fsize}{#1\fsize}\selectfont}%
  \ifx\svgwidth\undefined%
    \setlength{\unitlength}{1695.30721349bp}%
    \ifx\svgscale\undefined%
      \relax%
    \else%
      \setlength{\unitlength}{\unitlength * \real{\svgscale}}%
    \fi%
  \else%
    \setlength{\unitlength}{\svgwidth}%
  \fi%
  \global\let\svgwidth\undefined%
  \global\let\svgscale\undefined%
  \makeatother%
  \begin{picture}(1,0.3824645)%
    \lineheight{1}%
    \setlength\tabcolsep{0pt}%
    \put(0,0){\includegraphics[width=\unitlength,page=1]{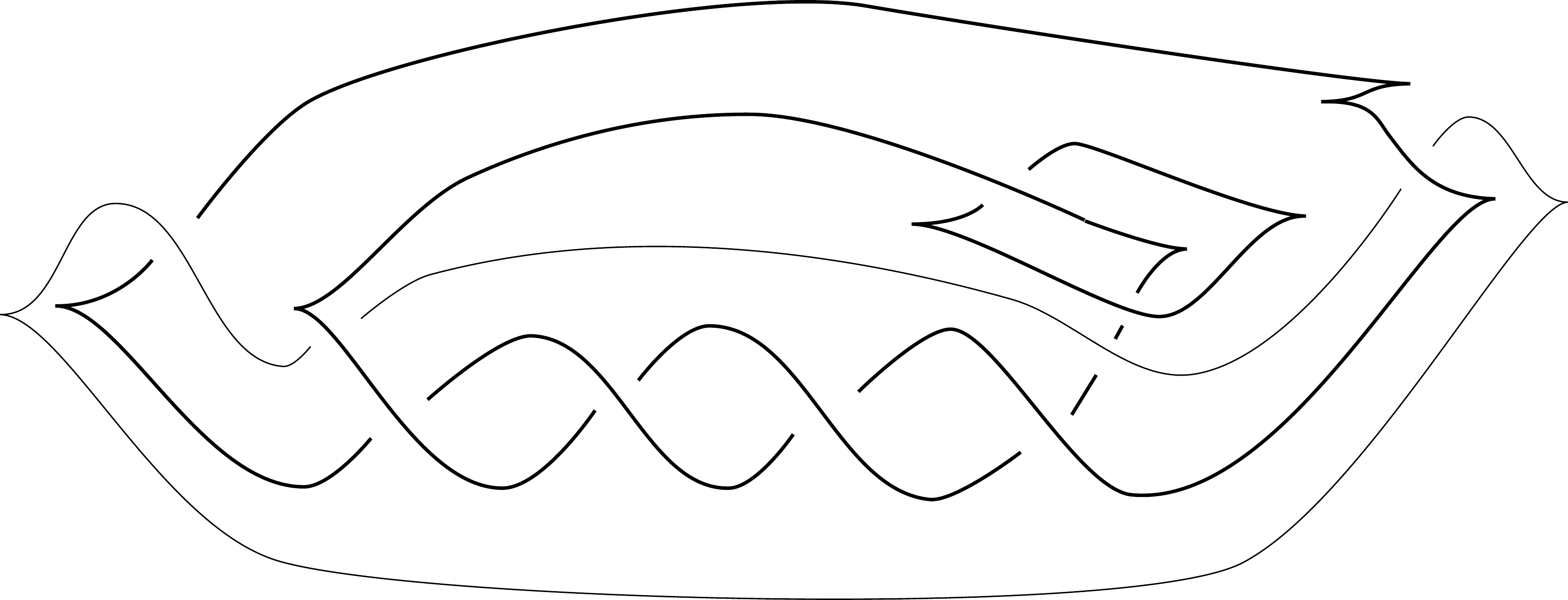}}%
    \put(0.90712044,0.10135155){\color[rgb]{0,0,0}\makebox(0,0)[lt]{\lineheight{1.25}\smash{\begin{tabular}[t]{l}$K_3$\end{tabular}}}}%
  \end{picture}%
\endgroup%
        \end{figure}
     \begin{figure}[t]
 \centering   
  \def\svgwidth{0.3\textwidth}
\begingroup%
  \makeatletter%
  \providecommand\color[2][]{%
    \errmessage{(Inkscape) Color is used for the text in Inkscape, but the package 'color.sty' is not loaded}%
    \renewcommand\color[2][]{}%
  }%
  \providecommand\transparent[1]{%
    \errmessage{(Inkscape) Transparency is used (non-zero) for the text in Inkscape, but the package 'transparent.sty' is not loaded}%
    \renewcommand\transparent[1]{}%
  }%
  \providecommand\rotatebox[2]{#2}%
  \newcommand*\fsize{\dimexpr\f@size pt\relax}%
  \newcommand*\lineheight[1]{\fontsize{\fsize}{#1\fsize}\selectfont}%
  \ifx\svgwidth\undefined%
    \setlength{\unitlength}{253.82315826bp}%
    \ifx\svgscale\undefined%
      \relax%
    \else%
      \setlength{\unitlength}{\unitlength * \real{\svgscale}}%
    \fi%
  \else%
    \setlength{\unitlength}{\svgwidth}%
  \fi%
  \global\let\svgwidth\undefined%
  \global\let\svgscale\undefined%
  \makeatother%
  \begin{picture}(1,0.85098618)%
    \lineheight{1}%
    \setlength\tabcolsep{0pt}%
    \put(0,0){\includegraphics[width=\unitlength,page=1]{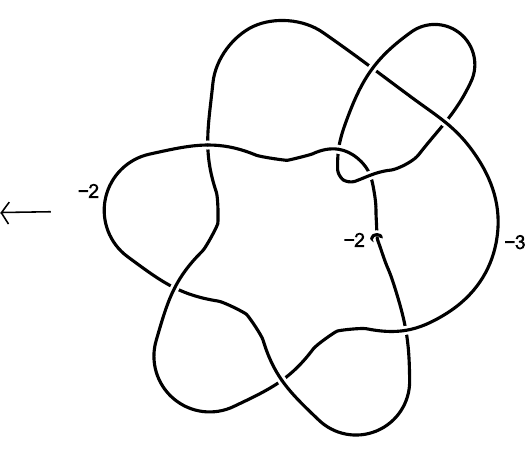}}%
    \put(0.88769399,0.63453992){\color[rgb]{0,0,0}\makebox(0,0)[lt]{\lineheight{1.25}\smash{\begin{tabular}[t]{l}$K_3$\end{tabular}}}}%
  \end{picture}%
\endgroup%
\vspace{0.5cm}
  \def\svgwidth{0.3\textwidth}
\begingroup%
  \makeatletter%
  \providecommand\color[2][]{%
    \errmessage{(Inkscape) Color is used for the text in Inkscape, but the package 'color.sty' is not loaded}%
    \renewcommand\color[2][]{}%
  }%
  \providecommand\transparent[1]{%
    \errmessage{(Inkscape) Transparency is used (non-zero) for the text in Inkscape, but the package 'transparent.sty' is not loaded}%
    \renewcommand\transparent[1]{}%
  }%
  \providecommand\rotatebox[2]{#2}%
  \newcommand*\fsize{\dimexpr\f@size pt\relax}%
  \newcommand*\lineheight[1]{\fontsize{\fsize}{#1\fsize}\selectfont}%
  \ifx\svgwidth\undefined%
    \setlength{\unitlength}{243.12890625bp}%
    \ifx\svgscale\undefined%
      \relax%
    \else%
      \setlength{\unitlength}{\unitlength * \real{\svgscale}}%
    \fi%
  \else%
    \setlength{\unitlength}{\svgwidth}%
  \fi%
  \global\let\svgwidth\undefined%
  \global\let\svgscale\undefined%
  \makeatother%
  \begin{picture}(1,0.8884176)%
    \lineheight{1}%
    \setlength\tabcolsep{0pt}%
    \put(0,0){\includegraphics[width=\unitlength,page=1]{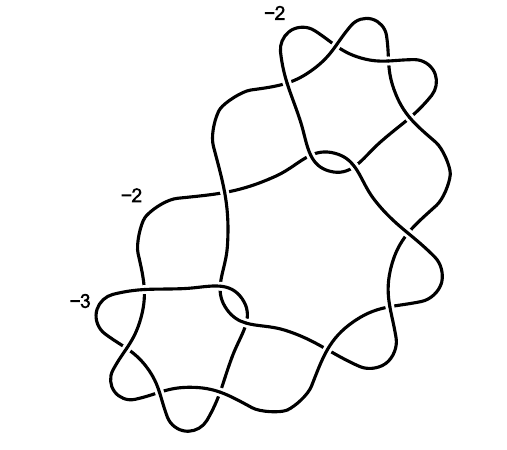}}%
    \put(0.90216605,0.53563008){\color[rgb]{0,0,0}\makebox(0,0)[lt]{\lineheight{1.25}\smash{\begin{tabular}[t]{l}$K_3$\end{tabular}}}}%
    \put(0,0){\includegraphics[width=\unitlength,page=2]{3-3.pdf}}%
  \end{picture}%
\endgroup%
    
  \def\svgwidth{0.27\textwidth}
\begingroup%
  \makeatletter%
  \providecommand\color[2][]{%
    \errmessage{(Inkscape) Color is used for the text in Inkscape, but the package 'color.sty' is not loaded}%
    \renewcommand\color[2][]{}%
  }%
  \providecommand\transparent[1]{%
    \errmessage{(Inkscape) Transparency is used (non-zero) for the text in Inkscape, but the package 'transparent.sty' is not loaded}%
    \renewcommand\transparent[1]{}%
  }%
  \providecommand\rotatebox[2]{#2}%
  \newcommand*\fsize{\dimexpr\f@size pt\relax}%
  \newcommand*\lineheight[1]{\fontsize{\fsize}{#1\fsize}\selectfont}%
  \ifx\svgwidth\undefined%
    \setlength{\unitlength}{213.49021912bp}%
    \ifx\svgscale\undefined%
      \relax%
    \else%
      \setlength{\unitlength}{\unitlength * \real{\svgscale}}%
    \fi%
  \else%
    \setlength{\unitlength}{\svgwidth}%
  \fi%
  \global\let\svgwidth\undefined%
  \global\let\svgscale\undefined%
  \makeatother%
  \begin{picture}(1,1.01175595)%
    \lineheight{1}%
    \setlength\tabcolsep{0pt}%
    \put(0,0){\includegraphics[width=\unitlength,page=1]{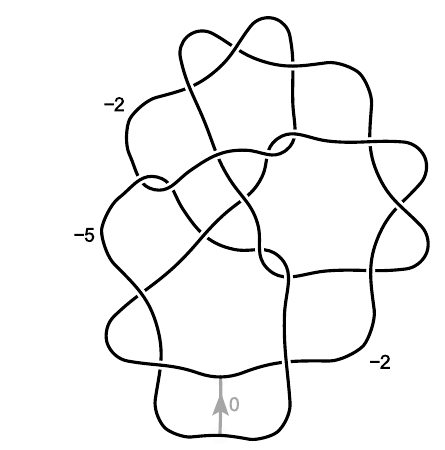}}%
    \put(0.82560764,0.85842786){\color[rgb]{0,0,0}\makebox(0,0)[lt]{\lineheight{1.25}\smash{\begin{tabular}[t]{l}$K_3$\end{tabular}}}}%
    \put(0,0){\includegraphics[width=\unitlength,page=2]{3-4.pdf}}%
  \end{picture}%
\endgroup%
      
\end{figure}
\begin{figure}[t]
 \centering
 \def\svgwidth{0.35\textwidth}
\begingroup%
  \makeatletter%
  \providecommand\color[2][]{%
    \errmessage{(Inkscape) Color is used for the text in Inkscape, but the package 'color.sty' is not loaded}%
    \renewcommand\color[2][]{}%
  }%
  \providecommand\transparent[1]{%
    \errmessage{(Inkscape) Transparency is used (non-zero) for the text in Inkscape, but the package 'transparent.sty' is not loaded}%
    \renewcommand\transparent[1]{}%
  }%
  \providecommand\rotatebox[2]{#2}%
  \newcommand*\fsize{\dimexpr\f@size pt\relax}%
  \newcommand*\lineheight[1]{\fontsize{\fsize}{#1\fsize}\selectfont}%
  \ifx\svgwidth\undefined%
    \setlength{\unitlength}{287.86956024bp}%
    \ifx\svgscale\undefined%
      \relax%
    \else%
      \setlength{\unitlength}{\unitlength * \real{\svgscale}}%
    \fi%
  \else%
    \setlength{\unitlength}{\svgwidth}%
  \fi%
  \global\let\svgwidth\undefined%
  \global\let\svgscale\undefined%
  \makeatother%
  \begin{picture}(1,0.72168302)%
    \lineheight{1}%
    \setlength\tabcolsep{0pt}%
    \put(0,0){\includegraphics[width=\unitlength,page=1]{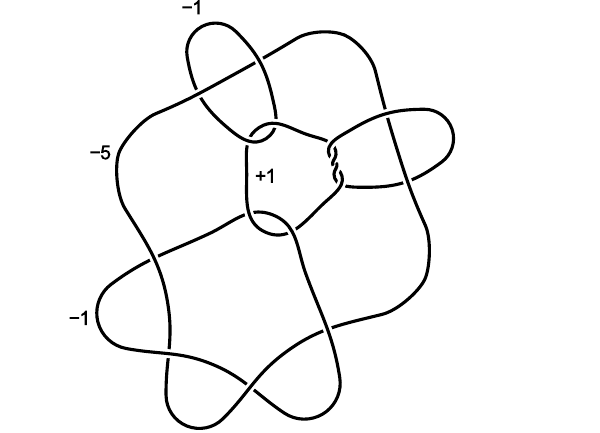}}%
    \put(0.68851303,0.56208565){\color[rgb]{0,0,0}\makebox(0,0)[lt]{\lineheight{1.25}\smash{\begin{tabular}[t]{l}$K_3\cong\gamma_+'$\end{tabular}}}}%
    \put(0,0){\includegraphics[width=\unitlength,page=2]{3-5.pdf}}%
  \end{picture}%
\endgroup%

 \def\svgwidth{0.25\textwidth}
\begingroup%
  \makeatletter%
  \providecommand\color[2][]{%
    \errmessage{(Inkscape) Color is used for the text in Inkscape, but the package 'color.sty' is not loaded}%
    \renewcommand\color[2][]{}%
  }%
  \providecommand\transparent[1]{%
    \errmessage{(Inkscape) Transparency is used (non-zero) for the text in Inkscape, but the package 'transparent.sty' is not loaded}%
    \renewcommand\transparent[1]{}%
  }%
  \providecommand\rotatebox[2]{#2}%
  \newcommand*\fsize{\dimexpr\f@size pt\relax}%
  \newcommand*\lineheight[1]{\fontsize{\fsize}{#1\fsize}\selectfont}%
  \ifx\svgwidth\undefined%
    \setlength{\unitlength}{212.206604bp}%
    \ifx\svgscale\undefined%
      \relax%
    \else%
      \setlength{\unitlength}{\unitlength * \real{\svgscale}}%
    \fi%
  \else%
    \setlength{\unitlength}{\svgwidth}%
  \fi%
  \global\let\svgwidth\undefined%
  \global\let\svgscale\undefined%
  \makeatother%
  \begin{picture}(1,1.01787596)%
    \lineheight{1}%
    \setlength\tabcolsep{0pt}%
    \put(0,0){\includegraphics[width=\unitlength,page=1]{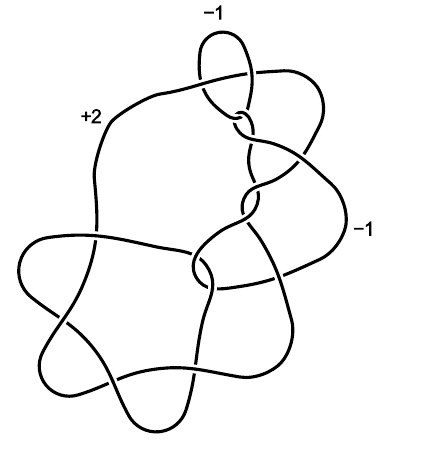}}%
    \put(0.04755872,0.50978077){\color[rgb]{0,0,0}\makebox(0,0)[lt]{\lineheight{1.25}\smash{\begin{tabular}[t]{l}$\gamma_+'$\end{tabular}}}}%
    \put(0,0){\includegraphics[width=\unitlength,page=2]{3-6.pdf}}%
  \end{picture}%
\endgroup%
 
 \def\svgwidth{0.35\textwidth}
\begingroup%
  \makeatletter%
  \providecommand\color[2][]{%
    \errmessage{(Inkscape) Color is used for the text in Inkscape, but the package 'color.sty' is not loaded}%
    \renewcommand\color[2][]{}%
  }%
  \providecommand\transparent[1]{%
    \errmessage{(Inkscape) Transparency is used (non-zero) for the text in Inkscape, but the package 'transparent.sty' is not loaded}%
    \renewcommand\transparent[1]{}%
  }%
  \providecommand\rotatebox[2]{#2}%
  \newcommand*\fsize{\dimexpr\f@size pt\relax}%
  \newcommand*\lineheight[1]{\fontsize{\fsize}{#1\fsize}\selectfont}%
  \ifx\svgwidth\undefined%
    \setlength{\unitlength}{328.61048126bp}%
    \ifx\svgscale\undefined%
      \relax%
    \else%
      \setlength{\unitlength}{\unitlength * \real{\svgscale}}%
    \fi%
  \else%
    \setlength{\unitlength}{\svgwidth}%
  \fi%
  \global\let\svgwidth\undefined%
  \global\let\svgscale\undefined%
  \makeatother%
  \begin{picture}(1,0.65731318)%
    \lineheight{1}%
    \setlength\tabcolsep{0pt}%
    \put(0,0){\includegraphics[width=\unitlength,page=1]{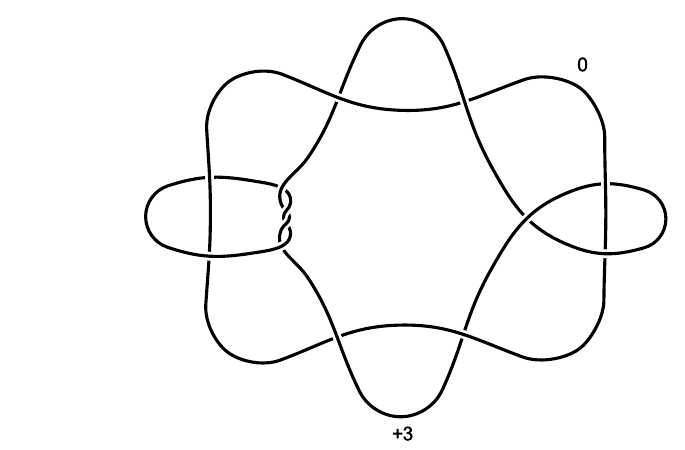}}%
    \put(-0.00161324,0.23270923){\color[rgb]{0,0,0}\makebox(0,0)[lt]{\lineheight{1.25}\smash{\begin{tabular}[t]{l}$\gamma_+'\cong\gamma_+$\end{tabular}}}}%
  \end{picture}%
\endgroup%
        
\end{figure}
\begin{figure}[t]
 \centering
 \def\svgwidth{0.302\textwidth}
\begingroup%
  \makeatletter%
  \providecommand\color[2][]{%
    \errmessage{(Inkscape) Color is used for the text in Inkscape, but the package 'color.sty' is not loaded}%
    \renewcommand\color[2][]{}%
  }%
  \providecommand\transparent[1]{%
    \errmessage{(Inkscape) Transparency is used (non-zero) for the text in Inkscape, but the package 'transparent.sty' is not loaded}%
    \renewcommand\transparent[1]{}%
  }%
  \providecommand\rotatebox[2]{#2}%
  \newcommand*\fsize{\dimexpr\f@size pt\relax}%
  \newcommand*\lineheight[1]{\fontsize{\fsize}{#1\fsize}\selectfont}%
  \ifx\svgwidth\undefined%
    \setlength{\unitlength}{257.89860535bp}%
    \ifx\svgscale\undefined%
      \relax%
    \else%
      \setlength{\unitlength}{\unitlength * \real{\svgscale}}%
    \fi%
  \else%
    \setlength{\unitlength}{\svgwidth}%
  \fi%
  \global\let\svgwidth\undefined%
  \global\let\svgscale\undefined%
  \makeatother%
  \begin{picture}(1,0.83753846)%
    \lineheight{1}%
    \setlength\tabcolsep{0pt}%
    \put(0,0){\includegraphics[width=\unitlength,page=1]{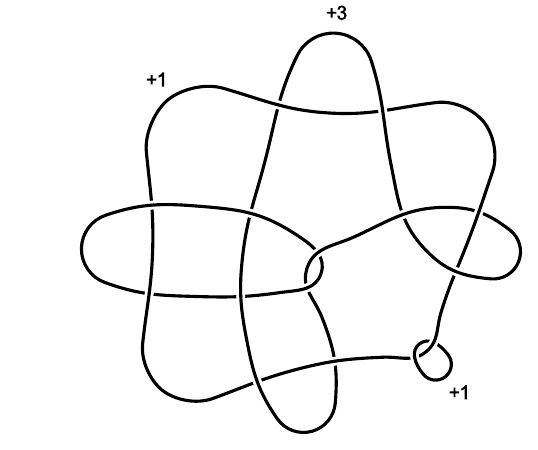}}%
    \put(0.14828999,0.49517565){\color[rgb]{0,0,0}\makebox(0,0)[lt]{\lineheight{1.25}\smash{\begin{tabular}[t]{l}$\gamma_+$\end{tabular}}}}%
    \put(0,0){\includegraphics[width=\unitlength,page=2]{3-8.pdf}}%
  \end{picture}%
\endgroup%

 \def\svgwidth{0.302\textwidth}
\begingroup%
  \makeatletter%
  \providecommand\color[2][]{%
    \errmessage{(Inkscape) Color is used for the text in Inkscape, but the package 'color.sty' is not loaded}%
    \renewcommand\color[2][]{}%
  }%
  \providecommand\transparent[1]{%
    \errmessage{(Inkscape) Transparency is used (non-zero) for the text in Inkscape, but the package 'transparent.sty' is not loaded}%
    \renewcommand\transparent[1]{}%
  }%
  \providecommand\rotatebox[2]{#2}%
  \newcommand*\fsize{\dimexpr\f@size pt\relax}%
  \newcommand*\lineheight[1]{\fontsize{\fsize}{#1\fsize}\selectfont}%
  \ifx\svgwidth\undefined%
    \setlength{\unitlength}{267.85668182bp}%
    \ifx\svgscale\undefined%
      \relax%
    \else%
      \setlength{\unitlength}{\unitlength * \real{\svgscale}}%
    \fi%
  \else%
    \setlength{\unitlength}{\svgwidth}%
  \fi%
  \global\let\svgwidth\undefined%
  \global\let\svgscale\undefined%
  \makeatother%
  \begin{picture}(1,0.80640139)%
    \lineheight{1}%
    \setlength\tabcolsep{0pt}%
    \put(0,0){\includegraphics[width=\unitlength,page=1]{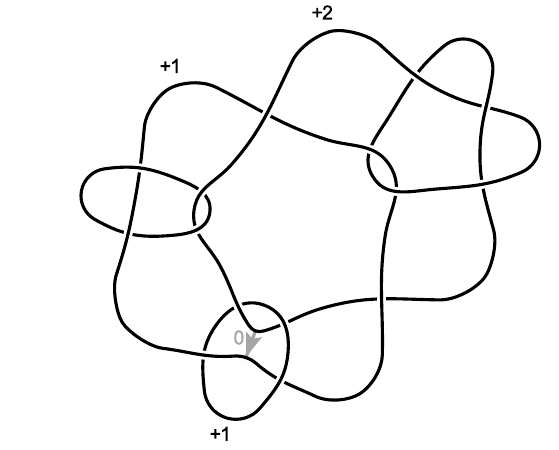}}%
    \put(0.15128468,0.54679774){\color[rgb]{0,0,0}\makebox(0,0)[lt]{\lineheight{1.25}\smash{\begin{tabular}[t]{l}$\gamma_+$\end{tabular}}}}%
    \put(0,0){\includegraphics[width=\unitlength,page=2]{3-9.pdf}}%
  \end{picture}%
\endgroup%
\hspace{2cm}    
 \def\svgwidth{0.27\textwidth}
\begingroup%
  \makeatletter%
  \providecommand\color[2][]{%
    \errmessage{(Inkscape) Color is used for the text in Inkscape, but the package 'color.sty' is not loaded}%
    \renewcommand\color[2][]{}%
  }%
  \providecommand\transparent[1]{%
    \errmessage{(Inkscape) Transparency is used (non-zero) for the text in Inkscape, but the package 'transparent.sty' is not loaded}%
    \renewcommand\transparent[1]{}%
  }%
  \providecommand\rotatebox[2]{#2}%
  \newcommand*\fsize{\dimexpr\f@size pt\relax}%
  \newcommand*\lineheight[1]{\fontsize{\fsize}{#1\fsize}\selectfont}%
  \ifx\svgwidth\undefined%
    \setlength{\unitlength}{297.64112091bp}%
    \ifx\svgscale\undefined%
      \relax%
    \else%
      \setlength{\unitlength}{\unitlength * \real{\svgscale}}%
    \fi%
  \else%
    \setlength{\unitlength}{\svgwidth}%
  \fi%
  \global\let\svgwidth\undefined%
  \global\let\svgscale\undefined%
  \makeatother%
  \begin{picture}(1,0.72570618)%
    \lineheight{1}%
    \setlength\tabcolsep{0pt}%
    \put(0,0){\includegraphics[width=\unitlength,page=1]{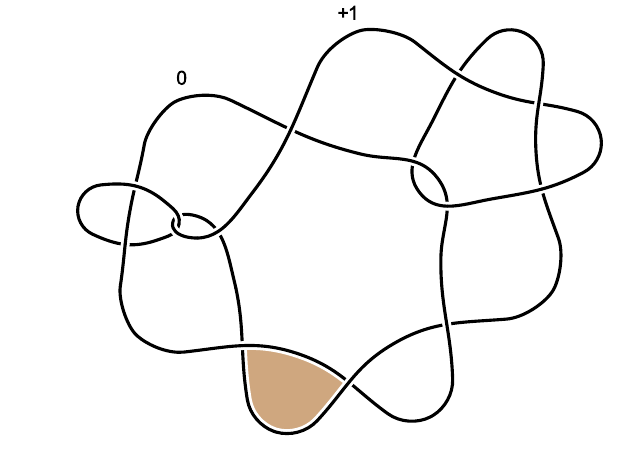}}%
    \put(0.12297441,0.46415827){\color[rgb]{0,0,0}\makebox(0,0)[lt]{\lineheight{1.25}\smash{\begin{tabular}[t]{l}$\gamma_+$\end{tabular}}}}%
    \put(0,0){\includegraphics[width=\unitlength,page=2]{3-10.pdf}}%
  \end{picture}%
\endgroup%
       
\end{figure} \clearpage\noindent
\begin{figure}[t]
 \centering
  \def\svgwidth{0.24\textwidth}
\begingroup%
  \makeatletter%
  \providecommand\color[2][]{%
    \errmessage{(Inkscape) Color is used for the text in Inkscape, but the package 'color.sty' is not loaded}%
    \renewcommand\color[2][]{}%
  }%
  \providecommand\transparent[1]{%
    \errmessage{(Inkscape) Transparency is used (non-zero) for the text in Inkscape, but the package 'transparent.sty' is not loaded}%
    \renewcommand\transparent[1]{}%
  }%
  \providecommand\rotatebox[2]{#2}%
  \newcommand*\fsize{\dimexpr\f@size pt\relax}%
  \newcommand*\lineheight[1]{\fontsize{\fsize}{#1\fsize}\selectfont}%
  \ifx\svgwidth\undefined%
    \setlength{\unitlength}{279.06058502bp}%
    \ifx\svgscale\undefined%
      \relax%
    \else%
      \setlength{\unitlength}{\unitlength * \real{\svgscale}}%
    \fi%
  \else%
    \setlength{\unitlength}{\svgwidth}%
  \fi%
  \global\let\svgwidth\undefined%
  \global\let\svgscale\undefined%
  \makeatother%
  \begin{picture}(1,0.77402547)%
    \lineheight{1}%
    \setlength\tabcolsep{0pt}%
    \put(0,0){\includegraphics[width=\unitlength,page=1]{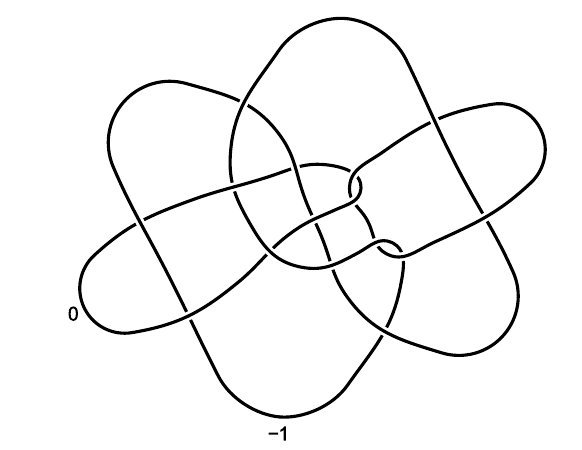}}%
    \put(0.80132368,0.62199656){\color[rgb]{0,0,0}\makebox(0,0)[lt]{\lineheight{1.25}\smash{\begin{tabular}[t]{l}$\gamma_+$\end{tabular}}}}%
    \put(0,0){\includegraphics[width=\unitlength,page=2]{3a2.pdf}}%
  \end{picture}%
\endgroup%

 \def\svgwidth{0.21\textwidth}
\begingroup%
  \makeatletter%
  \providecommand\color[2][]{%
    \errmessage{(Inkscape) Color is used for the text in Inkscape, but the package 'color.sty' is not loaded}%
    \renewcommand\color[2][]{}%
  }%
  \providecommand\transparent[1]{%
    \errmessage{(Inkscape) Transparency is used (non-zero) for the text in Inkscape, but the package 'transparent.sty' is not loaded}%
    \renewcommand\transparent[1]{}%
  }%
  \providecommand\rotatebox[2]{#2}%
  \newcommand*\fsize{\dimexpr\f@size pt\relax}%
  \newcommand*\lineheight[1]{\fontsize{\fsize}{#1\fsize}\selectfont}%
  \ifx\svgwidth\undefined%
    \setlength{\unitlength}{233.27861023bp}%
    \ifx\svgscale\undefined%
      \relax%
    \else%
      \setlength{\unitlength}{\unitlength * \real{\svgscale}}%
    \fi%
  \else%
    \setlength{\unitlength}{\svgwidth}%
  \fi%
  \global\let\svgwidth\undefined%
  \global\let\svgscale\undefined%
  \makeatother%
  \begin{picture}(1,0.92593144)%
    \lineheight{1}%
    \setlength\tabcolsep{0pt}%
    \put(0,0){\includegraphics[width=\unitlength,page=1]{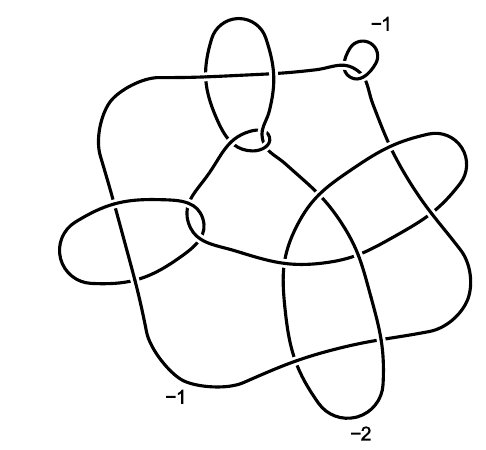}}%
    \put(0.15671876,0.26409994){\color[rgb]{0,0,0}\makebox(0,0)[lt]{\lineheight{1.25}\smash{\begin{tabular}[t]{l}$\gamma_+$\end{tabular}}}}%
    \put(0,0){\includegraphics[width=\unitlength,page=2]{3a3.pdf}}%
  \end{picture}%
\endgroup%

 \def\svgwidth{0.25\textwidth}
\begingroup%
  \makeatletter%
  \providecommand\color[2][]{%
    \errmessage{(Inkscape) Color is used for the text in Inkscape, but the package 'color.sty' is not loaded}%
    \renewcommand\color[2][]{}%
  }%
  \providecommand\transparent[1]{%
    \errmessage{(Inkscape) Transparency is used (non-zero) for the text in Inkscape, but the package 'transparent.sty' is not loaded}%
    \renewcommand\transparent[1]{}%
  }%
  \providecommand\rotatebox[2]{#2}%
  \newcommand*\fsize{\dimexpr\f@size pt\relax}%
  \newcommand*\lineheight[1]{\fontsize{\fsize}{#1\fsize}\selectfont}%
  \ifx\svgwidth\undefined%
    \setlength{\unitlength}{266.2625885bp}%
    \ifx\svgscale\undefined%
      \relax%
    \else%
      \setlength{\unitlength}{\unitlength * \real{\svgscale}}%
    \fi%
  \else%
    \setlength{\unitlength}{\svgwidth}%
  \fi%
  \global\let\svgwidth\undefined%
  \global\let\svgscale\undefined%
  \makeatother%
  \begin{picture}(1,0.76079106)%
    \lineheight{1}%
    \setlength\tabcolsep{0pt}%
    \put(0,0){\includegraphics[width=\unitlength,page=1]{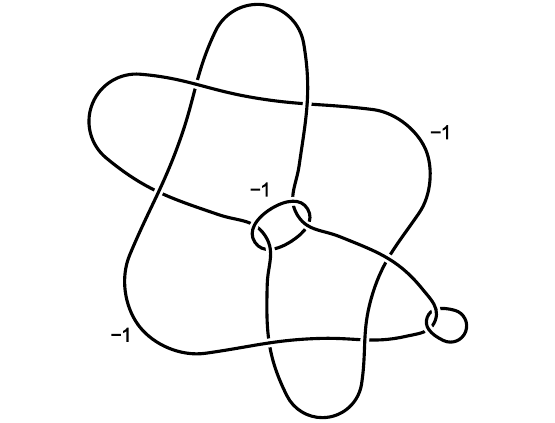}}%
    \put(0.80606887,0.22277645){\color[rgb]{0,0,0}\makebox(0,0)[lt]{\lineheight{1.25}\smash{\begin{tabular}[t]{l}$\gamma_+$\end{tabular}}}}%
    \put(0,0){\includegraphics[width=\unitlength,page=2]{3a4.pdf}}%
  \end{picture}%
\endgroup%

   \def\svgwidth{0.25\textwidth}
\begingroup%
  \makeatletter%
  \providecommand\color[2][]{%
    \errmessage{(Inkscape) Color is used for the text in Inkscape, but the package 'color.sty' is not loaded}%
    \renewcommand\color[2][]{}%
  }%
  \providecommand\transparent[1]{%
    \errmessage{(Inkscape) Transparency is used (non-zero) for the text in Inkscape, but the package 'transparent.sty' is not loaded}%
    \renewcommand\transparent[1]{}%
  }%
  \providecommand\rotatebox[2]{#2}%
  \newcommand*\fsize{\dimexpr\f@size pt\relax}%
  \newcommand*\lineheight[1]{\fontsize{\fsize}{#1\fsize}\selectfont}%
  \ifx\svgwidth\undefined%
    \setlength{\unitlength}{263.1988678bp}%
    \ifx\svgscale\undefined%
      \relax%
    \else%
      \setlength{\unitlength}{\unitlength * \real{\svgscale}}%
    \fi%
  \else%
    \setlength{\unitlength}{\svgwidth}%
  \fi%
  \global\let\svgwidth\undefined%
  \global\let\svgscale\undefined%
  \makeatother%
  \begin{picture}(1,0.82067222)%
    \lineheight{1}%
    \setlength\tabcolsep{0pt}%
    \put(0,0){\includegraphics[width=\unitlength,page=1]{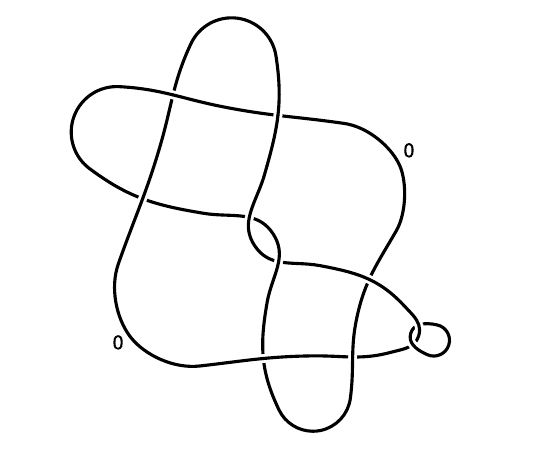}}%
    \put(0.80395388,0.24406605){\color[rgb]{0,0,0}\makebox(0,0)[lt]{\lineheight{1.25}\smash{\begin{tabular}[t]{l}$\gamma_+$\end{tabular}}}}%
    \put(0,0){\includegraphics[width=\unitlength,page=2]{3a5.pdf}}%
  \end{picture}%
\endgroup%

\end{figure}   
\begin{figure}[t]
 \centering
  \def\svgwidth{0.4\textwidth}
\begingroup%
  \makeatletter%
  \providecommand\color[2][]{%
    \errmessage{(Inkscape) Color is used for the text in Inkscape, but the package 'color.sty' is not loaded}%
    \renewcommand\color[2][]{}%
  }%
  \providecommand\transparent[1]{%
    \errmessage{(Inkscape) Transparency is used (non-zero) for the text in Inkscape, but the package 'transparent.sty' is not loaded}%
    \renewcommand\transparent[1]{}%
  }%
  \providecommand\rotatebox[2]{#2}%
  \newcommand*\fsize{\dimexpr\f@size pt\relax}%
  \newcommand*\lineheight[1]{\fontsize{\fsize}{#1\fsize}\selectfont}%
  \ifx\svgwidth\undefined%
    \setlength{\unitlength}{1101.37948104bp}%
    \ifx\svgscale\undefined%
      \relax%
    \else%
      \setlength{\unitlength}{\unitlength * \real{\svgscale}}%
    \fi%
  \else%
    \setlength{\unitlength}{\svgwidth}%
  \fi%
  \global\let\svgwidth\undefined%
  \global\let\svgscale\undefined%
  \makeatother%
  \begin{picture}(1,0.66828161)%
    \lineheight{1}%
    \setlength\tabcolsep{0pt}%
    \put(0,0){\includegraphics[width=\unitlength,page=1]{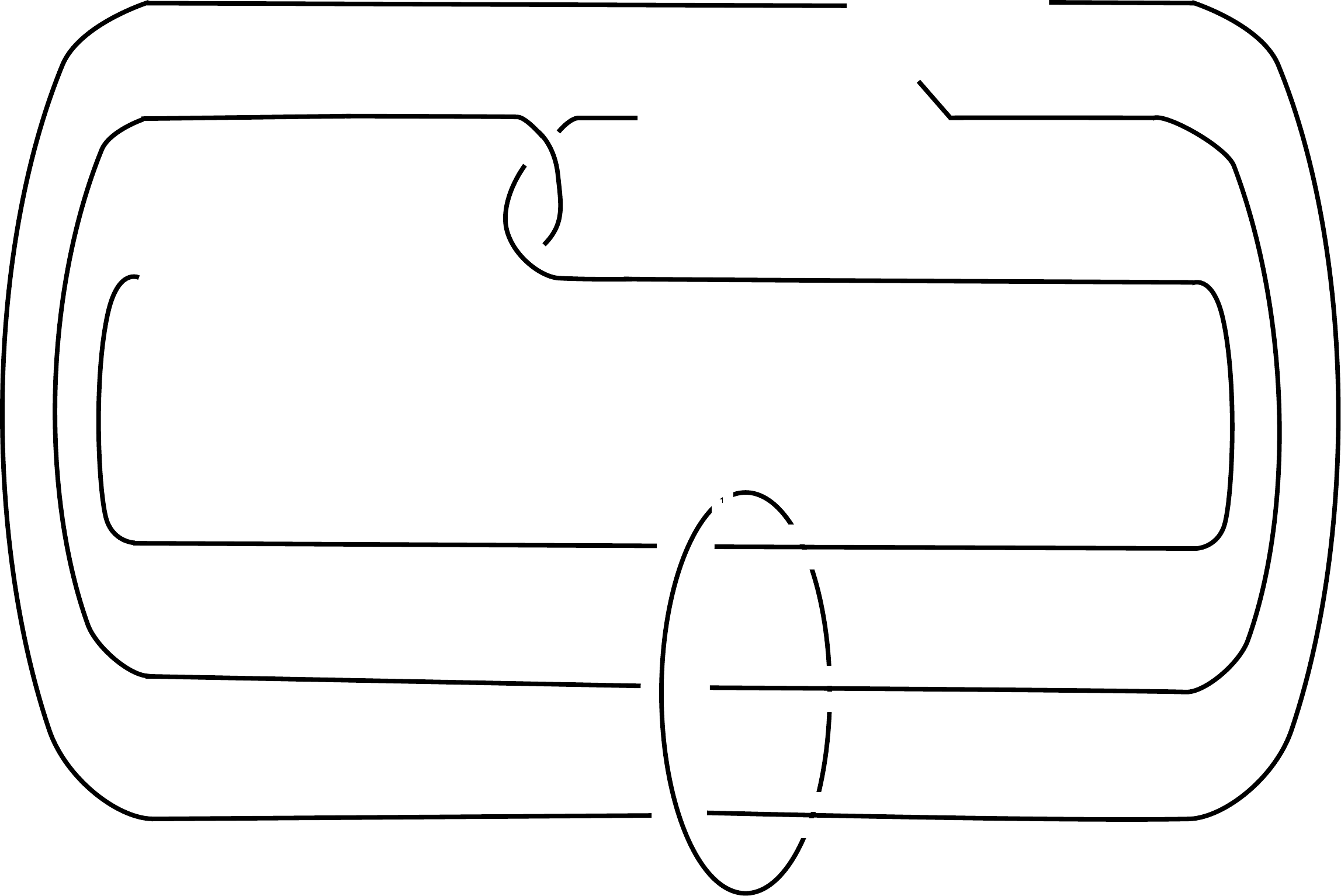}}%
    \put(0.16668633,0.27732651){\color[rgb]{0,0,0}\makebox(0,0)[lt]{\lineheight{1.25}\smash{\begin{tabular}[t]{l}$0$\end{tabular}}}}%
    \put(0.45043592,0.20316221){\color[rgb]{0,0,0}\makebox(0,0)[lt]{\lineheight{1.25}\smash{\begin{tabular}[t]{l}$0$\end{tabular}}}}%
    \put(0.57740686,0.35026791){\color[rgb]{0,0,0}\makebox(0,0)[lt]{\lineheight{1.25}\smash{\begin{tabular}[t]{l}$\gamma_+$\end{tabular}}}}%
    \put(0,0){\includegraphics[width=\unitlength,page=2]{3a1.pdf}}%
  \end{picture}%
\endgroup%

 \caption{\smaller[1]{The self-diffeomorphism $F$ of $Y^-(3)$ mapping $K_3$ to $\gamma_-$. There are four diffeomorphisms $f_A,f_B,f_C$ and $\Phi_{2,1}$ involved: we have that $f_A(\gamma_+')=K_3,$ $f_B(\gamma_+')=\gamma_+',$ $f_C(\gamma_+)=\gamma_+'$ while $\Phi_{2,1}(\gamma_-)=\gamma_+$ is the identification $Y^-(2,1)\cong-Y^+(0,0)$ established in Property 2) of Theorem \ref{teo:AK}. The diffeomorphism we want is $F=\Phi_{2,1}^{-1}\circ f_C^{-1}\circ f_B\circ f_A^{-1}$.}}
\label{Diffeo1}
\end{figure} 

 We start from the case of $m=3$, corresponding to $\Sigma(2,5,7)$:
 \[\tb_{\xist}(\bigcirc)-L^T\Lambda^{-1}L=-1 
 -\left(\begin{matrix}
  1 & -1 & 0 
 \end{matrix}\right)\left(\begin{matrix}
   -3 & -4 & -2\\
    -4 & -6 & -3 \\
    -2 & -3 & -2
 \end{matrix}\right)\left(\begin{matrix}
  1 \\ -1 \\ 0 
 \end{matrix}\right)
 =-1+1=0\] and \[L^T\Lambda^{-1}\mathbf V=\left(\begin{matrix}
  1 & -1 & 0 
 \end{matrix}\right)\left(\begin{matrix}
   -3 & -4 & -2\\
    -4 & -6 & -3 \\
    -2 & -3 & -2
 \end{matrix}\right)\left(\begin{matrix}
  -1 \\ 0 \\ 0 
 \end{matrix}\right)=-1\:;\] hence, from Proposition \ref{teo:AC} we obtain \[\tau_\xi(K_3)=\dfrac{0+(-1)+1}{2}=0\hspace{0.5cm}\text{ and }\hspace{0.5cm}\tau_{\overline\xi}(K_3)=\dfrac{0-(-1)+1}{2}=1\:.\]

\begin{figure}[t]
 \centering
 \def\svgwidth{0.426\textwidth}
\begingroup%
  \makeatletter%
  \providecommand\color[2][]{%
    \errmessage{(Inkscape) Color is used for the text in Inkscape, but the package 'color.sty' is not loaded}%
    \renewcommand\color[2][]{}%
  }%
  \providecommand\transparent[1]{%
    \errmessage{(Inkscape) Transparency is used (non-zero) for the text in Inkscape, but the package 'transparent.sty' is not loaded}%
    \renewcommand\transparent[1]{}%
  }%
  \providecommand\rotatebox[2]{#2}%
  \newcommand*\fsize{\dimexpr\f@size pt\relax}%
  \newcommand*\lineheight[1]{\fontsize{\fsize}{#1\fsize}\selectfont}%
  \ifx\svgwidth\undefined%
    \setlength{\unitlength}{1678.78291093bp}%
    \ifx\svgscale\undefined%
      \relax%
    \else%
      \setlength{\unitlength}{\unitlength * \real{\svgscale}}%
    \fi%
  \else%
    \setlength{\unitlength}{\svgwidth}%
  \fi%
  \global\let\svgwidth\undefined%
  \global\let\svgscale\undefined%
  \makeatother%
  \begin{picture}(1,0.44061252)%
    \lineheight{1}%
    \setlength\tabcolsep{0pt}%
    \put(0,0){\includegraphics[width=\unitlength,page=1]{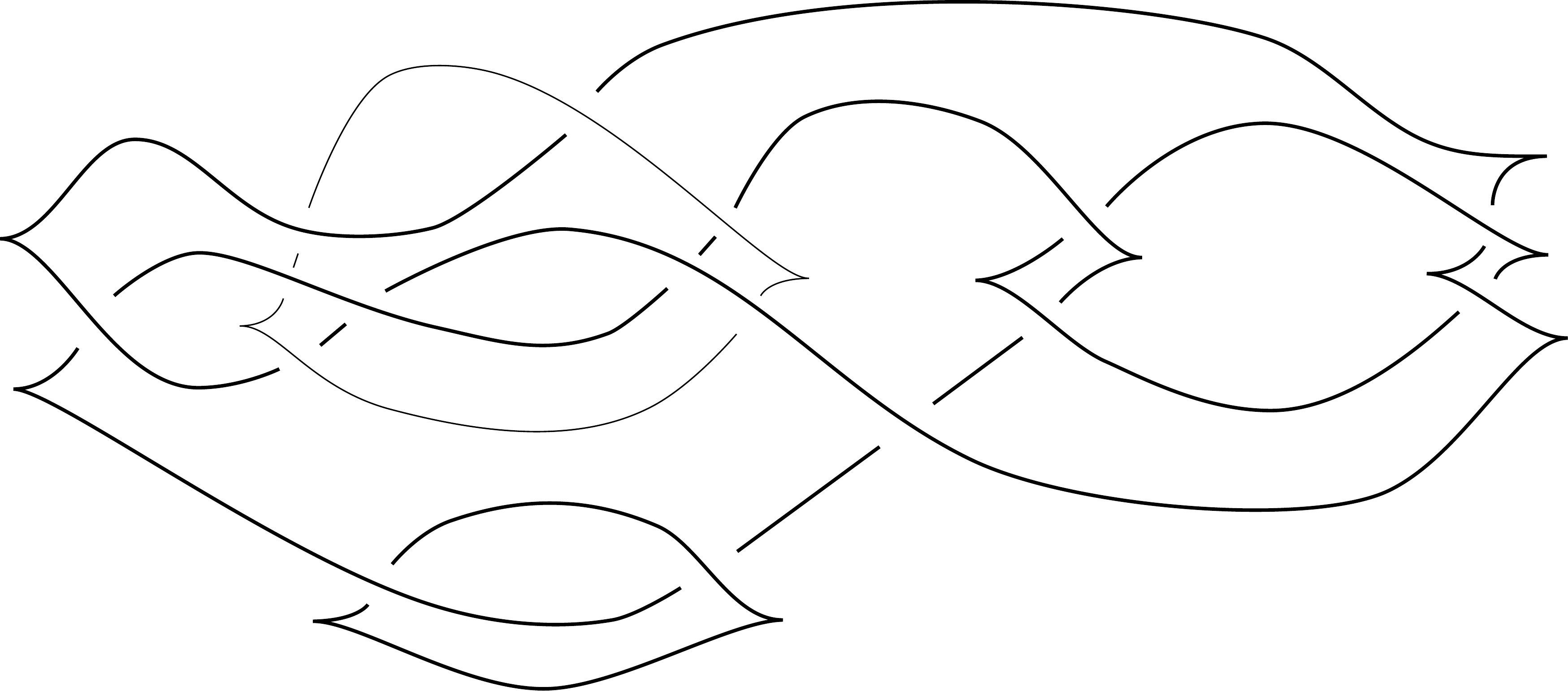}}%
    \put(0.22955183,0.41291554){\color[rgb]{0,0,0}\makebox(0,0)[lt]{\lineheight{1.25}\smash{\begin{tabular}[t]{l}$K_4$\end{tabular}}}}%
  \end{picture}%
\endgroup%
\hspace{6cm}
 \def\svgwidth{0.25\textwidth}
\begingroup%
  \makeatletter%
  \providecommand\color[2][]{%
    \errmessage{(Inkscape) Color is used for the text in Inkscape, but the package 'color.sty' is not loaded}%
    \renewcommand\color[2][]{}%
  }%
  \providecommand\transparent[1]{%
    \errmessage{(Inkscape) Transparency is used (non-zero) for the text in Inkscape, but the package 'transparent.sty' is not loaded}%
    \renewcommand\transparent[1]{}%
  }%
  \providecommand\rotatebox[2]{#2}%
  \newcommand*\fsize{\dimexpr\f@size pt\relax}%
  \newcommand*\lineheight[1]{\fontsize{\fsize}{#1\fsize}\selectfont}%
  \ifx\svgwidth\undefined%
    \setlength{\unitlength}{213.14373779bp}%
    \ifx\svgscale\undefined%
      \relax%
    \else%
      \setlength{\unitlength}{\unitlength * \real{\svgscale}}%
    \fi%
  \else%
    \setlength{\unitlength}{\svgwidth}%
  \fi%
  \global\let\svgwidth\undefined%
  \global\let\svgscale\undefined%
  \makeatother%
  \begin{picture}(1,1.01340064)%
    \lineheight{1}%
    \setlength\tabcolsep{0pt}%
    \put(0,0){\includegraphics[width=\unitlength,page=1]{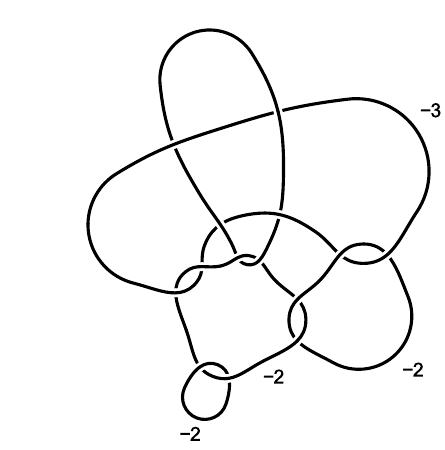}}%
    \put(0.61217663,0.86247664){\color[rgb]{0,0,0}\makebox(0,0)[lt]{\lineheight{1.25}\smash{\begin{tabular}[t]{l}$K_4$\end{tabular}}}}%
    \put(0,0){\includegraphics[width=\unitlength,page=2]{4-2.pdf}}%
  \end{picture}%
\endgroup%

 \def\svgwidth{0.32\textwidth}
\begingroup%
  \makeatletter%
  \providecommand\color[2][]{%
    \errmessage{(Inkscape) Color is used for the text in Inkscape, but the package 'color.sty' is not loaded}%
    \renewcommand\color[2][]{}%
  }%
  \providecommand\transparent[1]{%
    \errmessage{(Inkscape) Transparency is used (non-zero) for the text in Inkscape, but the package 'transparent.sty' is not loaded}%
    \renewcommand\transparent[1]{}%
  }%
  \providecommand\rotatebox[2]{#2}%
  \newcommand*\fsize{\dimexpr\f@size pt\relax}%
  \newcommand*\lineheight[1]{\fontsize{\fsize}{#1\fsize}\selectfont}%
  \ifx\svgwidth\undefined%
    \setlength{\unitlength}{265.08867645bp}%
    \ifx\svgscale\undefined%
      \relax%
    \else%
      \setlength{\unitlength}{\unitlength * \real{\svgscale}}%
    \fi%
  \else%
    \setlength{\unitlength}{\svgwidth}%
  \fi%
  \global\let\svgwidth\undefined%
  \global\let\svgscale\undefined%
  \makeatother%
  \begin{picture}(1,0.81482168)%
    \lineheight{1}%
    \setlength\tabcolsep{0pt}%
    \put(0,0){\includegraphics[width=\unitlength,page=1]{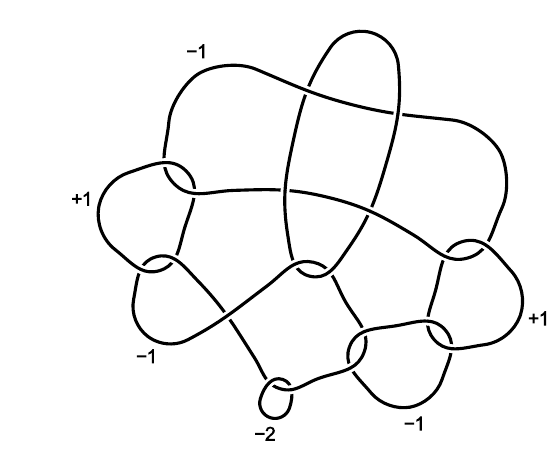}}%
    \put(0.73858266,0.68600114){\color[rgb]{0,0,0}\makebox(0,0)[lt]{\lineheight{1.25}\smash{\begin{tabular}[t]{l}$K_4$\end{tabular}}}}%
    \put(0,0){\includegraphics[width=\unitlength,page=2]{4-3.pdf}}%
  \end{picture}%
\endgroup%
    
 \def\svgwidth{0.25\textwidth}
\begingroup%
  \makeatletter%
  \providecommand\color[2][]{%
    \errmessage{(Inkscape) Color is used for the text in Inkscape, but the package 'color.sty' is not loaded}%
    \renewcommand\color[2][]{}%
  }%
  \providecommand\transparent[1]{%
    \errmessage{(Inkscape) Transparency is used (non-zero) for the text in Inkscape, but the package 'transparent.sty' is not loaded}%
    \renewcommand\transparent[1]{}%
  }%
  \providecommand\rotatebox[2]{#2}%
  \newcommand*\fsize{\dimexpr\f@size pt\relax}%
  \newcommand*\lineheight[1]{\fontsize{\fsize}{#1\fsize}\selectfont}%
  \ifx\svgwidth\undefined%
    \setlength{\unitlength}{211.77163696bp}%
    \ifx\svgscale\undefined%
      \relax%
    \else%
      \setlength{\unitlength}{\unitlength * \real{\svgscale}}%
    \fi%
  \else%
    \setlength{\unitlength}{\svgwidth}%
  \fi%
  \global\let\svgwidth\undefined%
  \global\let\svgscale\undefined%
  \makeatother%
  \begin{picture}(1,1.01996662)%
    \lineheight{1}%
    \setlength\tabcolsep{0pt}%
    \put(0,0){\includegraphics[width=\unitlength,page=1]{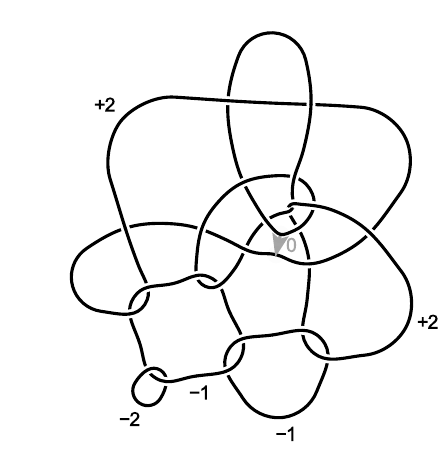}}%
    \put(0.71091201,0.89681755){\color[rgb]{0,0,0}\makebox(0,0)[lt]{\lineheight{1.25}\smash{\begin{tabular}[t]{l}$K_4$\end{tabular}}}}%
    \put(0,0){\includegraphics[width=\unitlength,page=2]{4-4.pdf}}%
  \end{picture}%
\endgroup%
      
\end{figure}
\begin{figure}[t]
 \centering
 \def\svgwidth{0.25\textwidth}
\begingroup%
  \makeatletter%
  \providecommand\color[2][]{%
    \errmessage{(Inkscape) Color is used for the text in Inkscape, but the package 'color.sty' is not loaded}%
    \renewcommand\color[2][]{}%
  }%
  \providecommand\transparent[1]{%
    \errmessage{(Inkscape) Transparency is used (non-zero) for the text in Inkscape, but the package 'transparent.sty' is not loaded}%
    \renewcommand\transparent[1]{}%
  }%
  \providecommand\rotatebox[2]{#2}%
  \newcommand*\fsize{\dimexpr\f@size pt\relax}%
  \newcommand*\lineheight[1]{\fontsize{\fsize}{#1\fsize}\selectfont}%
  \ifx\svgwidth\undefined%
    \setlength{\unitlength}{229.30741882bp}%
    \ifx\svgscale\undefined%
      \relax%
    \else%
      \setlength{\unitlength}{\unitlength * \real{\svgscale}}%
    \fi%
  \else%
    \setlength{\unitlength}{\svgwidth}%
  \fi%
  \global\let\svgwidth\undefined%
  \global\let\svgscale\undefined%
  \makeatother%
  \begin{picture}(1,0.94196691)%
    \lineheight{1}%
    \setlength\tabcolsep{0pt}%
    \put(0,0){\includegraphics[width=\unitlength,page=1]{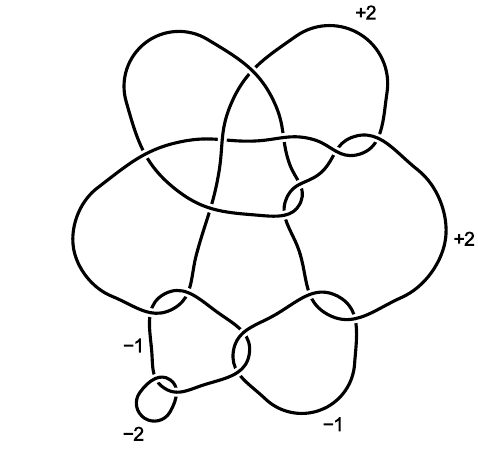}}%
    \put(0.43631915,0.88709157){\color[rgb]{0,0,0}\makebox(0,0)[lt]{\lineheight{1.25}\smash{\begin{tabular}[t]{l}$K_4$\end{tabular}}}}%
    \put(0,0){\includegraphics[width=\unitlength,page=2]{4-5.pdf}}%
  \end{picture}%
\endgroup%

 \def\svgwidth{0.22\textwidth}
\begingroup%
  \makeatletter%
  \providecommand\color[2][]{%
    \errmessage{(Inkscape) Color is used for the text in Inkscape, but the package 'color.sty' is not loaded}%
    \renewcommand\color[2][]{}%
  }%
  \providecommand\transparent[1]{%
    \errmessage{(Inkscape) Transparency is used (non-zero) for the text in Inkscape, but the package 'transparent.sty' is not loaded}%
    \renewcommand\transparent[1]{}%
  }%
  \providecommand\rotatebox[2]{#2}%
  \newcommand*\fsize{\dimexpr\f@size pt\relax}%
  \newcommand*\lineheight[1]{\fontsize{\fsize}{#1\fsize}\selectfont}%
  \ifx\svgwidth\undefined%
    \setlength{\unitlength}{203.94179535bp}%
    \ifx\svgscale\undefined%
      \relax%
    \else%
      \setlength{\unitlength}{\unitlength * \real{\svgscale}}%
    \fi%
  \else%
    \setlength{\unitlength}{\svgwidth}%
  \fi%
  \global\let\svgwidth\undefined%
  \global\let\svgscale\undefined%
  \makeatother%
  \begin{picture}(1,1.05912572)%
    \lineheight{1}%
    \setlength\tabcolsep{0pt}%
    \put(0,0){\includegraphics[width=\unitlength,page=1]{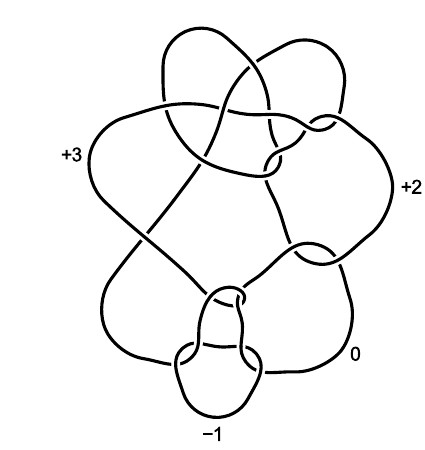}}%
    \put(0.54384989,0.98910599){\color[rgb]{0,0,0}\makebox(0,0)[lt]{\lineheight{1.25}\smash{\begin{tabular}[t]{l}$K_4$\end{tabular}}}}%
    \put(0,0){\includegraphics[width=\unitlength,page=2]{4-6.pdf}}%
  \end{picture}%
\endgroup%
  
 \def\svgwidth{0.27\textwidth}
\begingroup%
  \makeatletter%
  \providecommand\color[2][]{%
    \errmessage{(Inkscape) Color is used for the text in Inkscape, but the package 'color.sty' is not loaded}%
    \renewcommand\color[2][]{}%
  }%
  \providecommand\transparent[1]{%
    \errmessage{(Inkscape) Transparency is used (non-zero) for the text in Inkscape, but the package 'transparent.sty' is not loaded}%
    \renewcommand\transparent[1]{}%
  }%
  \providecommand\rotatebox[2]{#2}%
  \newcommand*\fsize{\dimexpr\f@size pt\relax}%
  \newcommand*\lineheight[1]{\fontsize{\fsize}{#1\fsize}\selectfont}%
  \ifx\svgwidth\undefined%
    \setlength{\unitlength}{264.10350037bp}%
    \ifx\svgscale\undefined%
      \relax%
    \else%
      \setlength{\unitlength}{\unitlength * \real{\svgscale}}%
    \fi%
  \else%
    \setlength{\unitlength}{\svgwidth}%
  \fi%
  \global\let\svgwidth\undefined%
  \global\let\svgscale\undefined%
  \makeatother%
  \begin{picture}(1,0.81786118)%
    \lineheight{1}%
    \setlength\tabcolsep{0pt}%
    \put(0,0){\includegraphics[width=\unitlength,page=1]{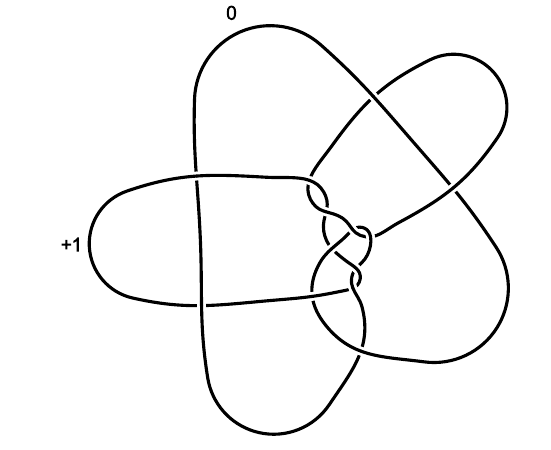}}%
    \put(0.90993578,0.69489933){\color[rgb]{0,0,0}\makebox(0,0)[lt]{\lineheight{1.25}\smash{\begin{tabular}[t]{l}$K_4$\end{tabular}}}}%
    \put(0,0){\includegraphics[width=\unitlength,page=2]{4-7.pdf}}%
  \end{picture}%
\endgroup%
      
\end{figure}  
\begin{figure}[t]
 \centering
  \def\svgwidth{0.35\textwidth}
\begingroup%
  \makeatletter%
  \providecommand\color[2][]{%
    \errmessage{(Inkscape) Color is used for the text in Inkscape, but the package 'color.sty' is not loaded}%
    \renewcommand\color[2][]{}%
  }%
  \providecommand\transparent[1]{%
    \errmessage{(Inkscape) Transparency is used (non-zero) for the text in Inkscape, but the package 'transparent.sty' is not loaded}%
    \renewcommand\transparent[1]{}%
  }%
  \providecommand\rotatebox[2]{#2}%
  \newcommand*\fsize{\dimexpr\f@size pt\relax}%
  \newcommand*\lineheight[1]{\fontsize{\fsize}{#1\fsize}\selectfont}%
  \ifx\svgwidth\undefined%
    \setlength{\unitlength}{336.56190491bp}%
    \ifx\svgscale\undefined%
      \relax%
    \else%
      \setlength{\unitlength}{\unitlength * \real{\svgscale}}%
    \fi%
  \else%
    \setlength{\unitlength}{\svgwidth}%
  \fi%
  \global\let\svgwidth\undefined%
  \global\let\svgscale\undefined%
  \makeatother%
  \begin{picture}(1,0.64178386)%
    \lineheight{1}%
    \setlength\tabcolsep{0pt}%
    \put(0,0){\includegraphics[width=\unitlength,page=1]{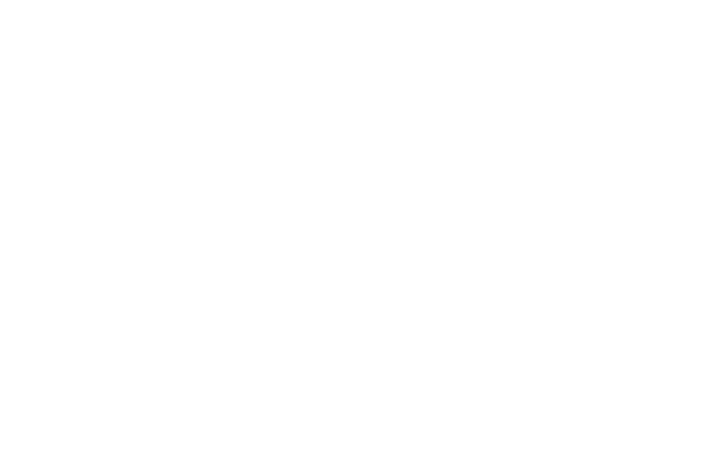}}%
    \put(0.54381404,0.55635853){\color[rgb]{0,0,0}\makebox(0,0)[lt]{\lineheight{1.25}\smash{\begin{tabular}[t]{l}$K_4\cong\gamma_+'\cong\gamma_+$\end{tabular}}}}%
    \put(0,0){\includegraphics[width=\unitlength,page=2]{4a3.pdf}}%
  \end{picture}%
\endgroup%

  \def\svgwidth{0.25\textwidth}
\begingroup%
  \makeatletter%
  \providecommand\color[2][]{%
    \errmessage{(Inkscape) Color is used for the text in Inkscape, but the package 'color.sty' is not loaded}%
    \renewcommand\color[2][]{}%
  }%
  \providecommand\transparent[1]{%
    \errmessage{(Inkscape) Transparency is used (non-zero) for the text in Inkscape, but the package 'transparent.sty' is not loaded}%
    \renewcommand\transparent[1]{}%
  }%
  \providecommand\rotatebox[2]{#2}%
  \newcommand*\fsize{\dimexpr\f@size pt\relax}%
  \newcommand*\lineheight[1]{\fontsize{\fsize}{#1\fsize}\selectfont}%
  \ifx\svgwidth\undefined%
    \setlength{\unitlength}{288.87691498bp}%
    \ifx\svgscale\undefined%
      \relax%
    \else%
      \setlength{\unitlength}{\unitlength * \real{\svgscale}}%
    \fi%
  \else%
    \setlength{\unitlength}{\svgwidth}%
  \fi%
  \global\let\svgwidth\undefined%
  \global\let\svgscale\undefined%
  \makeatother%
  \begin{picture}(1,0.7477233)%
    \lineheight{1}%
    \setlength\tabcolsep{0pt}%
    \put(0,0){\includegraphics[width=\unitlength,page=1]{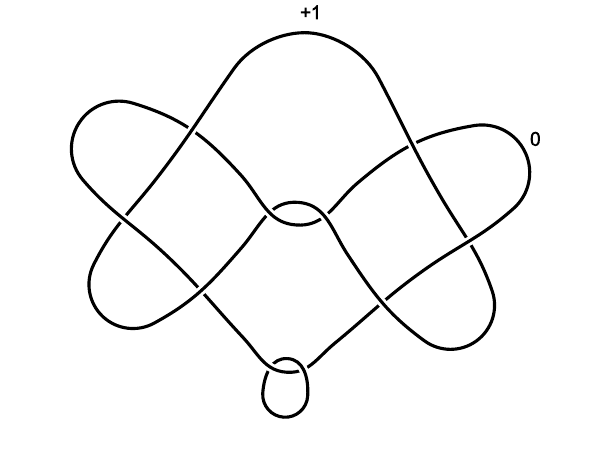}}%
    \put(0.53209235,0.09157255){\color[rgb]{0,0,0}\makebox(0,0)[lt]{\lineheight{1.25}\smash{\begin{tabular}[t]{l}$\gamma_+$\end{tabular}}}}%
    \put(0,0){\includegraphics[width=\unitlength,page=2]{4a2.pdf}}%
  \end{picture}%
\endgroup%
 \hspace{2cm}
  \def\svgwidth{0.34\textwidth}
\begingroup%
  \makeatletter%
  \providecommand\color[2][]{%
    \errmessage{(Inkscape) Color is used for the text in Inkscape, but the package 'color.sty' is not loaded}%
    \renewcommand\color[2][]{}%
  }%
  \providecommand\transparent[1]{%
    \errmessage{(Inkscape) Transparency is used (non-zero) for the text in Inkscape, but the package 'transparent.sty' is not loaded}%
    \renewcommand\transparent[1]{}%
  }%
  \providecommand\rotatebox[2]{#2}%
  \newcommand*\fsize{\dimexpr\f@size pt\relax}%
  \newcommand*\lineheight[1]{\fontsize{\fsize}{#1\fsize}\selectfont}%
  \ifx\svgwidth\undefined%
    \setlength{\unitlength}{1101.37948104bp}%
    \ifx\svgscale\undefined%
      \relax%
    \else%
      \setlength{\unitlength}{\unitlength * \real{\svgscale}}%
    \fi%
  \else%
    \setlength{\unitlength}{\svgwidth}%
  \fi%
  \global\let\svgwidth\undefined%
  \global\let\svgscale\undefined%
  \makeatother%
  \begin{picture}(1,0.66828161)%
    \lineheight{1}%
    \setlength\tabcolsep{0pt}%
    \put(0,0){\includegraphics[width=\unitlength,page=1]{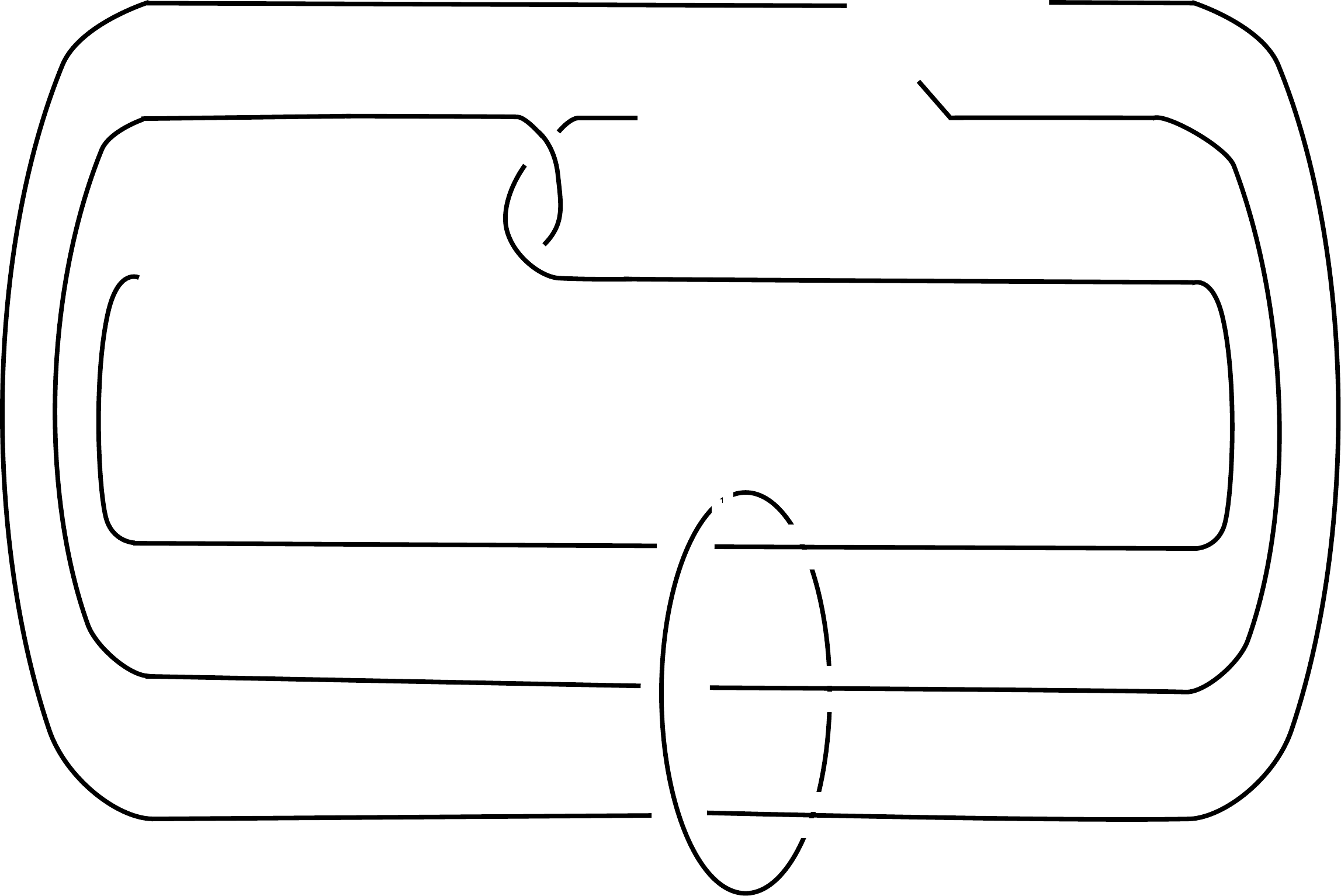}}%
    \put(0.16668633,0.27732651){\color[rgb]{0,0,0}\makebox(0,0)[lt]{\lineheight{1.25}\smash{\begin{tabular}[t]{l}$1$\end{tabular}}}}%
    \put(0.45043592,0.20316221){\color[rgb]{0,0,0}\makebox(0,0)[lt]{\lineheight{1.25}\smash{\begin{tabular}[t]{l}$0$\end{tabular}}}}%
    \put(0.57740686,0.35026791){\color[rgb]{0,0,0}\makebox(0,0)[lt]{\lineheight{1.25}\smash{\begin{tabular}[t]{l}$\gamma_+$\end{tabular}}}}%
    \put(0,0){\includegraphics[width=\unitlength,page=2]{4a1.pdf}}%
  \end{picture}%
\endgroup%

 \caption{\smaller[1]{The self-diffeomorphism $F$ of $Y^-(4)$ mapping $K_4$ to $\gamma_-$. There are three diffeomorphisms $f_D,f_E$ and $\Phi_{2,2}$ involved: we have that $f_D(K_4)=\gamma_+',$ $f_E(\gamma_+)=\gamma_+'$ while $\Phi_{2,2}(\gamma_-)=\gamma_+$ is the identification $Y^-(2,2)\cong-Y^+(0,-1)$ established in Property 2) of Theorem \ref{teo:AK}. The diffeomorphism we want is $F=\Phi_{2,2}^{-1}\circ f_E^{-1}\circ f_D$.}}
\label{Diffeo2}
\end{figure} 

 We continue with the case of $m=4$, corresponding to $\Sigma(3,4,5)$:
 \[\tb_{\xist}(\bigcirc)-L^T\Lambda^{-1}L=-1- 
 \left(\begin{matrix}
  -2 & 0 & 1 & 0
 \end{matrix}\right)\left(\begin{matrix}
   -4 & -5 & -6 & -3\\
    -5 & -7 & -8 & -4 \\
    -6 & -8 & -10 & -5 \\
    -3 & -4 & -5 & -3
 \end{matrix}\right)\left(\begin{matrix}
  -2 \\ 0 \\ 1 \\ 0 
 \end{matrix}\right)
 =-1+2=1\] and \[L^T\Lambda^{-1}\mathbf V=\left(\begin{matrix}
  -2 & 0 & 1 & 0
 \end{matrix}\right)\left(\begin{matrix}
   -4 & -5 & -6 & -3\\
    -5 & -7 & -8 & -4 \\
    -6 & -8 & -10 & -5 \\
    -3 & -4 & -5 & -3
 \end{matrix}\right)\left(\begin{matrix}
  -1 \\ 0 \\ 0 \\ 0 
 \end{matrix}\right)=-2\:;\] hence, in the same way as before we obtain \[\tau_\xi(K_4)=0\hspace{0.5cm}\text{ and }\hspace{0.5cm}\tau_{\overline\xi}(K_4)=2\:.\]
 The fact that $\mathcal T(m)=\{\tau_\xi(K_m),\tau_{\overline\xi}(K_m)\}$ follows from the discussion at the beginning of the section.
\end{proof}

\section{Symplectic structures and Heegaard Floer cobordism maps}
\label{section:three}
We are going to show that $W^-(m)$ does not carry a symplectic structure for $m=2,3,4$, and then extend this result to every $Z\#W^-(m)$ with $Z$ closed 4-manifold. 
We obtain our obstruction by computing the cobordism map induced by $W^-(m)$ in Heegaard Floer homology, and studying its behavior. 

We recall that the duality identification $\widehat{HF}(-Y,\s)\simeq\widehat{HF}(Y,\s)^\bullet$ holds for every 3-manifold $(Y,\s)$. Using the same terminology as in \cite[Subsection 2.4]{CM-negative} and \cite[Section 1]{OSz-fullpath}, we write $\rho_*:\widehat{HF}(Y,\s)\rightarrow\Ker U\subset HF^+(Y,\s)$ for the map induced by the inclusion, and $\widehat{HF}^{\text{ev}}(Y,\s)\subset\widehat{HF}(Y,\s)$ for the subspace of the homology classes with Maslov grading $d$ such that $d-d(Y,\s)\equiv0\text{ mod }2$, where $d(Y,\s)$ is the correction term. Thus, whenever $Y$ is presented by an almost-rational plumbing tree the following isomorphisms hold: 
\begin{equation}
 \widehat{HF}^\text{ev}_*(Y,\s)\simeq\widehat{HF}^\text{ev}_{-*}(-Y,\s)\simeq\Ker U\cap HF^+_{-*}(-Y,\s)\:,   
 \label{eq:iso}
\end{equation} where the right-most relation is given by $\rho_*$. In addition, the subgroup $\Ker U\cap HF^+(-Y,\s)$ contains the (unique) homology class $\Theta^+_\s\in U^n\cdot HF^+(-Y,\s)$ for every $n\geq1$; according to Equation \eqref{eq:iso}, we can identify $\Theta^+_s$ with a distinguished non-zero class in $\widehat{HF}^\text{ev}(-Y,\s)$. 

The assumptions for Equation \eqref{eq:iso} are satisfied by $Y^-(m)$ for $m=2,3,4$, as these manifolds are canonically oriented Brieskorn spheres, see Figure \ref{Graph}. Using the full path algorithm of Ozsv\'ath and Szab\'o \cite{OSz-fullpath}, we can then compute their Heegaard Floer groups. We obtain \[\widehat{HF}_0(Y^-(m))\simeq\widehat{HF}_0(-Y^-(m))\simeq\widehat{HF}^{\text{ev}}(Y^-(m))\simeq\F^3_0\hspace{0.5cm}\text{ for }\hspace{0.5cm}m=2,3,4\:,\] where $\F$ is the field with two elements. We write the canonical basis, given by (the full path class of) characteristic vectors as described in \cite{OSz-fullpath,CM-negative}, of such an $\F$-vector space as $\mathcal B=\{[V_1],[V_0],[V_{-1}]\}$ in order to highlight the action of the conjugation involution $\mathcal J$; in fact, we have that $\mathcal J[V_i]=[V_{-i}]$ for each $i$.
\begin{prop}
 \label{prop:cobordism}
 The cobordism map $\widehat F_{W^-(m)}:\widehat{HF}(S^3)\rightarrow\widehat{HF}_0(Y^-(m))$ sends the generator to $[V_0]$ for $m=3,4$.   
\end{prop}
\begin{proof}
 Since $W^-(m)$ is an integral homology 4-ball, the image of the map $F^-_{W^-(m)}$ is a non-torsion homogeneous class in $HF^-_0(Y^-(m))$, see \cite{OSz-negative}. For this reason, because of the commutation of the cobordism maps with the projection $\psi_*:HF^-\rightarrow\widehat{HF}$, we have that $\theta=\widehat F_{W^-(m)}(1)$ is a combination of some of the elements of the basis $\mathcal B$; in addition, we see immediately that the number of such elements is odd. The reason is that if we take the dual map $\widehat F_{\overline{W^-(m)}}=\widehat F_{W^-(m)}^\bullet:\widehat{HF}_0(-Y^-(m))\rightarrow\widehat{HF}(S^3)$, induced by $W^-(m)$ taken upside-down, then $\widehat F_{\overline{W^-(m)}}(\Theta^+)=1$ and by \cite[Subsection 2.4]{CM-negative} we know that $\Theta^+=T_{[V_1]}+T_{[V_0]}+T_{[V_{-1}]}$, where $T_{[V_i]}$ is the functional that sends $[V_i]$ to one and the others to zero. Hence, if $\theta$ were the sum of two elements in $\mathcal B$ then $\widehat F_{\overline{W^-(m)}}(\Theta^+)=0$.

 Taking the involution $\mathcal J$ into account yields \[ \mathcal J\theta=\mathcal J\widehat F_{W^-(m)}(1)=\widehat F_{W^-(m)}(\mathcal J1)=\theta\:,\] because there is a unique $\Spin^c$-structure on $W^-(m)$. We are left with two options: either one has $\theta=[V_0]$ or $\theta=[V_1]+[V_0]+[V_{-1}]$; note that this is true also when $m=2$. Suppose that the second case is true; then the definition of $\tau_\eta(K)$ and the orientation-reversing symmetry property in \cite[Lemma 2.2]{AC} imply \[\tau_\eta(K)=-\tau_{\widehat c(\eta)}(K)=\min\big\{\tau_\gamma(K)\: : \: \gamma\in \widehat{HF}_{d_3(\eta)}(Y,\s_\eta)\text{ and }\langle\gamma,\widehat c(\eta)\rangle=1\big\}\] for every knot $K\subset Y$ and $\eta$ contact structure with non-vanishing invariant. Thus combining this fact, which is summarized in \cite[Theorem 1.7]{AC}, with Hedden and Raoux's relative adjunction inequality \cite[Theorem 1]{HR} yields \[\max\mathcal T(m)\leq\tau_\theta(W^-(2,m-2),\gamma_-)\leq g_4(W^-(2,m-2),\gamma_-)\:.\] This holds because we know for sure that \[\langle\theta,\widehat c(\xi)\rangle=\langle\theta,\mathcal J\widehat c(\xi)\rangle=\langle\theta,\widehat c(\overline\xi)\rangle=1\;,\] as because $\xi$ is the Milnor fillable structure on $Y^-(m)$ and we can take the characteristic vector $[V_1]$ in the way that $T_{[V_1]}=\widehat c(\xi)$ by \cite[Proposition 5.2]{CM-negative}.
 This leads to a contradiction: the knot $\gamma_-$ is smoothly slice in $W^-(2,m-2)$ because of Property 3) in Theorem \ref{teo:AK}, thus $g_4(W^-(2,m-2),\gamma_-)=0$; meanwhile, the set $\mathcal T(m)$ contains a positive integer for both the values of $m$ from Proposition \ref{prop:tau}, thus $\max\mathcal T(m)>0$. 
\end{proof}
 
\begin{lemma}
 \label{lemma:summand}
 The manifold $W^-(m)$ is not a symplectic filling for $m=2,3,4$.   
\end{lemma}
\begin{proof} 
 The behavior of $\widehat F_{W^-(m)}$ is sufficient to obstruct the existence of a symplectic structure on the Mazur manifold $W^-(m)$ when $m=3,4$; in fact, the naturality of the contact invariant under cobordism maps induced by a strong symplectic filling \cite{OSz-contact,G-fillability} implies $\widehat F_{\overline{W^-(m)}}(\widehat c(\xi))=\widehat F^\bullet_{W^-(m)}(T_{[V_1]})=1$, but this contradicts the fact that $\widehat F_{W^-(m)}(1)=[V_0]$ which we established in Proposition \ref{prop:cobordism}. 

 For $Y^-(2)\cong\Sigma(2,3,13)$ then the claim follows from the proof of \cite[Theorem 1.8]{MT2}. In fact, Mark and Tosun show that the closed 4-manifold obtained by gluing the plumbing $P_2$, represented by the standard graph of $-Y^-(2)$ (this is shown in \cite[Figure 3]{MT2}), to $W^-(2)$ is diffeomorphic to the elliptic surface $E(1)$; moreover, they mention that $E(1)$ has vanishing Seiberg-Witten invariant $SW^0$. Note that this manifold is well-defined by \cite[Theorem 1.1]{Akbulut}, such an argument was not needed in \cite[Theorem 1.8]{MT2} 
 
 Conversely, if $W^-(2)$ carried a symplectic structure then one could glue $W^-(2)$ to $P_2$ to obtain a 4-manifold with non-zero $SW^0$; this is the content of \cite[Lemma 3.1]{MT2}.
\end{proof}

We can now prove our main result.

\begin{proof}[Proof of Theorem \ref{teo:main}]
 After proving Lemma \ref{lemma:summand}, we just need to take care of the connected sum. For $m=3,4$ we just observe that the cobordism map $\widehat F_{Z,\s}$ induced by $(Z,\s)$ is either the identity or zero, while $\widehat F_{Z\#W^-(m),\:\s}=\widehat F_{W^-(m)}\circ\widehat F_{Z,\s}$ by the composition formula.

 For $m=2$, the cap construction in the proof of \cite[Lemma 3.1]{MT2}, together with \cite[Theorem 1.7]{MT2} and \cite[Theorem 1.1]{Akbulut}, gives a closed symplectic manifold $(M,\omega)$ with \[M=P_2\cup\bigl(Z\#W^-(2)\bigr) \cong Z\#\bigl(P_2\cup W^-(2)\bigr) \cong Z\#E(1)\:.\]
Let $F\subset P_2$ be the symplectic torus used in that construction, and put $f=\operatorname{PD}_{E(1)}[F]$; thus $f\neq0,\:f\cdot f=0$ and $\int_F\omega>0$. Write $\s_\omega$ for the $\Spin^c$-structure induced by an $\omega$-compatible almost-complex structure, so that $c_1(\s_\omega)=-K_\omega$; the canonical class of the cap is $f\lvert_{P_2}$. Since the restriction $H^2(E(1);\mathbb Z)\rightarrow H^2(P_2;\mathbb Z)$ is an isomorphism, the connected sum decomposition gives 
$$ H^2(M;\mathbb Z) \cong H^2(E(1);\mathbb Z)\oplus H^2(Z;\mathbb Z)\hspace{0.5cm}\text{ and }\hspace{0.5cm}c_1(\mathfrak s_\omega)=(-f,c_Z) $$
for some $c_Z\in H^2(Z;\mathbb Z)$.

If $b_2^+(Z)>0$ then the connected sum vanishing theorem, see \cite[Theorem 1.1]{Salamon2}, contradicts \cite[Theorem 2.2]{LiLiu}; hence we may assume that $b_2^+(Z)=0$, so that $b_2^+(M)=1$. Let $\mathcal P_\omega$ be the component of the positive cone containing $[\omega]$; we use the chamber convention of \cite{MT2}: for a metric $g$ whose normalized period point $h_g$ lies in $\mathcal P_\omega$, the zero-perturbation pair $(g,0)$ lies in the negative chamber whenever $c_1(\s_\omega)\cdot h_g<0$. In this convention one has $\operatorname{SW}^-(M,\s_\omega)=\pm1$.

Choose a metric $g_E$ on $E(1)$ with positive scalar curvature and a metric $g_Z$ on $Z$, and let $g_\epsilon$ be standard connected sum metrics whose neck stretches as $\epsilon\to0$. Let $h_\varepsilon\in\mathcal P_\omega$ be their period points, normalized by $h_\epsilon\cdot h_\epsilon=1$. Since $b_2^+(Z)=0$, these converge to $(h_E,0)$, where $h_E$ is the normalized period point of $g_E$. The sign of $h_E$ is determined by $f\cdot h_E>0$: indeed, $f\cdot f=0$ and $f\cdot[\omega]>0$, so the light-cone lemma \cite[Lemma 3.1]{LiLiu} places $f$ in the closure of $\mathcal P_\omega$. Consequently, we have that $$ c_1(\s_\omega)\cdot h_\epsilon \longrightarrow -f\cdot h_E<0$$ thus $(g_\varepsilon,0)$ lies in the negative chamber for all sufficiently small $\varepsilon$.
On $(E(1),g_E)$, the unperturbed Seiberg-Witten moduli space for the $\Spin^c$-structure $\s_\omega\lvert_{E(1)}$ is empty. Positive scalar curvature excludes irreducible solutions, while a reducible solution would require $(-f)^+=0$, contradicting $-f\cdot h_E<0$, where $(-f)^+$ is the self-dual component of $-f$ with respect to $g_E$. 

The standard arguments in \cite[Section 4]{Salamon2} and \cite[Theorem 8.6]{Salamon} now imply that the unperturbed moduli space on $(M,g_\epsilon)$ is empty for all sufficiently small $\epsilon$. Indeed, a sequence of solutions with $\epsilon\to0$ would limit to a finite-energy solution on the punctured $E(1)$-summand, which extends across the puncture to a solution on $(E(1),g_E)$ by \cite[Theorem 1.2]{Salamon2}. Therefore, we conclude that $\operatorname{SW}^-(M,\mathfrak s_\omega)=0$ contradicting \cite[Theorem 2.2]{LiLiu} and completing the proof.
\end{proof}

We note that the arguments given in the proofs above are not generalizable to other Brieskorn spheres, or even to other rational homology ball symplectic fillings of $Y^-(m)$; they strictly rely on the properties of either the slice knot $\gamma_-$ or the Mazur manifold $W^-(2)$.

We observe that the cobordism map $\widehat F_{-W^-(m)}:\widehat{HF}(S^3)\rightarrow\widehat{HF}_0(-Y^-(m))$, which we obtain by reversing the orientation of the 4-manifold, behaves in a remarkably different way.

\begin{prop}
 \label{prop:opposite}
 Suppose that $Y$ is presented by an almost-rational plumbing tree, and is the boundary of a rational homology ball $W$. Then the cobordism map $\widehat F_{-W,\mathfrak u}:\widehat{HF}(S^3)\rightarrow\widehat{HF}_0(-Y,\s)$ where $\s=\mathfrak u\lvert_{-Y}$ sends the generator to $\Theta^+_\s$.   
\end{prop}
\begin{proof}
 By \cite[Theorem 1.2]{OSz-fullpath} and \cite[Theorem 8.3]{Nemethi0} the subspace of $HF^+(Y,\s)$ supported in even degrees coincides with $\Imm\pi_*$ where $\pi_*:HF^\infty(Y,\s)\rightarrow HF^+(Y,\s)$, and by duality this implies that every homogeneous non-torsion element of $HF^-(-Y,\s)$ is in a canonical $\F[U]$-tower whose generator $\alpha$ has Maslov grading zero, while every other homogeneous class in $HF^-(-Y,\s)$ is torsion and with odd degree. 

 This means that $F^-_{-W,\mathfrak u}$ maps the generator to $\alpha$. The commutativity of the cobordism maps tells us that the projection of $\alpha$ into $\widehat{HF}_0(-Y,\s)$ is then sent to $\Theta^+_\s$, and we conclude by using the identification in Equation \eqref{eq:iso}. 
\end{proof}

Contrary to what happens in Proposition \ref{prop:cobordism}, the assumptions for this result are much broader. Note that it follows from \cite[Theorem 1.4]{ACM} that oppositely oriented small Brieskorn spheres do not admit any rational homology ball symplectic filling.

\section{Exotic 4-manifolds}
\label{section:four}
We can use Mazur manifolds which have Brieskorn spheres as boundary to produce examples of exotic pairs with definite-intersection form. 
We recall that the manifold $W^-(m)$, whose boundary $Y^-(m)$ is diffeomorphic to $\Sigma(2,3,13),$ $\Sigma(2,5,7)$ and $\Sigma(3,4,5)$ when $m=2,3$ and $4$ respectively, is contained in this family; moreover, many other Brieskorn spheres are as such, see \cite{Savk}. Nonetheless, the construction we apply to prove Theorem \ref{teo:exotic} works for more general 3-manifolds; in particular, we only need a $Y$ bounding a Mazur manifold $W$ and a negative-definite Stein domain $X$ with $\pi_1(X)=1$ and positive second Betti number.



Let $X_1:=X$, while $X_2:=\widehat X\#W$ where $\widehat X$ is any closed 4-manifold obtained by gluing $-W$ to $X$; finally, let $X_3:=\#m\C P^2\#W$ where $m=b_2(X)$. 

\begin{proof}[Proof of Theorem \ref{teo:exotic}]
 Since $X$ is negative-definite, the manifolds $X_1$ and $X_i$ for $i=2,3$ have the same intersection form by Donaldson's Theorem A. Moreover, we have that $\widehat X$ is simply connected, because $X$ and $-W$ are made only by 0-, 1- and 2-handles while $\pi_1(W)=1$; hence, the two manifolds are both simply connected. Freedman's classification \cite{Freedman}, together with its generalization due to Boyer \cite[Corollary 0.9]{Boyer}, then implies that $X_1$ is homeomorphic to $X_i$, as clearly $\partial X_1\cong\partial X_i=Y$ is an integral homology 3-sphere.



 We recall that from the proof of Proposition \ref{prop:cobordism} one has that $\widehat F_W(1)=\theta\in\widehat{HF}_0(Y)$ is self-conjugate, which means $\J\theta=\theta$; this holds whenever $W$ is an integral homology 4-ball. Since a closed simply connected 4-manifold whose intersection form is negative-definite can always be seen as a cobordism from $S^3$ to $S^3$, and its induced Heegaard Floer map is either zero or the identity \cite{OSz-negative}, we have that $\widehat F_{X_i,\mathfrak u}(1)\in\{0,\theta\}$ for $i=2,3$ and every $\mathfrak u\in\Spin^c(X_i)$. 
 
 In order to complete the proof, we just need to show that any Stein structure $J$ on $X_1$ is such that $\alpha:=\widehat F_{X_1,J}(1)$ is not self-conjugate. Since $Y$ bounds a Mazur manifold, the intersection form of $X_1$ is diagonalizable, and then $c_1(J)\neq0$ which implies that $J$ and $-J$ induce distinct $\Spin^c$-structures. We then know from Plamenevskaya \cite[Theorem 4]{Olga} that $\widehat F_{\overline X_1,J}(\widehat c(J\lvert_Y))=1$, where $\widehat F_{\overline X_1,J}:\widehat{HF}(-Y)\rightarrow\widehat{HF}(S^3)$ is the cobordism map obtained by taking $X_1$ upside-down and $\widehat c$ is the Ozsv\'ath-Szab\'o contact invariant, while $\widehat F_{\overline X_1,J}(\widehat c(-J\lvert_Y))=\widehat F_{\overline X_1,-J}(\widehat c(J\lvert_Y))=0$.
 


 By duality and since $\widehat F_{X_1,J}=(\widehat F_{\overline X_1,J})^\bullet$, we conclude that $\langle\alpha,\widehat c(J\lvert_Y)\rangle=1$ while \[0=\langle\alpha,\widehat c(-J\lvert_Y)\rangle=\langle\alpha,\J\widehat c(J\lvert_Y)\rangle=\langle\J\alpha,\widehat c(J\lvert_Y)\rangle\:,\] thus $\J\alpha\neq\alpha$.
\end{proof}


We conclude the paper by proving Corollary \ref{cor:exotic}; namely, that there exist two topologically isotopic smooth embeddings $f$ and $g$ of a Brieskorn sphere $Y$, in a simply connected 4-manifold $M$ with definite intersection form, such that no diffeomorphism of $M$ maps $f(Y)$ to $g(Y)$.

\begin{proof}[Proof of Corollary \ref{cor:exotic}]
 We take as $M$ any 4-manifold $-\widehat X$ as above, where $\partial X=Y$ is a Brieskorn sphere. We have that $W\subset D^4\subset M$; hence, we denote by $f(Y)$ the boundary of $W$ in $D^4$, while by $g(Y)$ the boundary of $-X$, that is the 3-manifold where we glue $W$ and $-X$ to get $M$. If a diffeomorphism of $M$ into itself would map $f(Y)$ to $g(Y)$ then their complements would also be diffeomorphic, and thus $X\cong\widehat X\#W$ which is impossible because of Theorem \ref{teo:exotic}.

 It remains to show that $f(Y)$ and $g(Y)$ are topologically isotopic. 
 Let $W_0$ be the copy of $W$ in a small $D^4$ and $W_1$ the other one, and fix an orientation-preserving identification $\phi:W_0\to W_1$. Write $C_i=\overline{M\setminus W_i}$, then $C_0$ is homeomorphic to $-\widehat X\#-W$ while $C_1=-X$, the inclusions $j_i:C_i\hookrightarrow M$ induce isometries on their second homology groups. 
 
 By \cite[Theorem 0.7 and Proposition 0.8(i)]{Boyer} there is a homeomorphism $F:C_0\to C_1$ extending $\phi|_{\partial W_0}$ and realizing the isometry $F_*=(j_1)_*^{-1}(j_0)_*$ from $H_2(C_0;\Z)$ to $H_2(C_1;\Z)$; the compatibility conditions are automatic because $Y$ is an integral homology sphere. Gluing $F$ to $\phi$ gives an orientation-preserving homeomorphism $H:M\to M$ with $H(W_0)=W_1$ and $H_*=\Id_{H_2(M;\Z)}$. By a result of Quinn \cite[Theorem 1.1]{Quinn}, see also \cite{GG} for a complete proof, such a homeomorphism is topologically isotopic to $\Id_M$.
\end{proof}


\begin{thebibliography}{9}
  \bibitem{Akbulut} S. Akbulut, \emph{Cork twists and automorphisms of $3$-manifolds},
    J. G\"okova Geom. Topol. GGT., \textbf{14} (2020), pp. 91--103.
 \bibitem{AD} S. Akbulut and S. Durusoy, \emph{An involution acting nontrivially on Heegaard-Floer homology},
  Geometry and topology of manifolds. Papers from the conference held at McMaster University, Hamilton, ON, Canada, May 14--18, 2004.
 \bibitem{Cork} S. Akbulut and \c C. Karakurt, \emph{Action of the Cork twist on Floer homology},
  Proceedings of the 18th G\"okova geometry-topology conference, G\"okova, Turkey, May 30--June 4, 2011.
 \bibitem{AK} S. Akbulut and R. Kirby, \emph{Mazur manifolds},
  Michigan Math. J., \textbf{26} (1979), no. 3, pp. 259--284. 
 \bibitem{AL} S. Akbulut and K. Larson, \emph{Brieskorn spheres bounding rational balls}, 
  Proc. Amer. Math. Soc., \textbf{146} (2018), no. 4, pp. 1817--1824.
  \bibitem{AC} A. Alfieri and A. Cavallo, \emph{Holomorphic curves in Stein domains and the tau-invariant},
   arXiv:2310.08657.
 \bibitem{ACM} A. Alfieri, A. Cavallo and I. Matkovi\v c, \emph{Brieskorn spheres and rational homology ball symplectic fillings},
   arXiv:2605.13812. 
 \bibitem{Boyer} S. Boyer, \emph{Simply-connected $4$-manifolds with a given boundary},
  Trans. Amer. Math. Soc., \textbf{298} (1986), no. 1, pp. 331--357. 
  \bibitem{Baptiste} B. Chantraine, \emph{Lagrangian concordance of Legendrian knots},
 Algebr. Geom. Topol., \textbf{10} (2010), no. 1, pp. 63--85. 
  \bibitem{CM-negative} A. Cavallo and I. Matkovi\v c, \emph{Fillable structures on negative-definite Seifert fibred spaces},
     arXiv:2604.28174.      
  \bibitem{DHM} I. Dai, M. Hedden and A. Mallick, \emph{Corks, involutions, and Heegaard Floer homology},
    J. Eur. Math. Soc., \textbf{25} (2023), no. 6, pp. 2319--2389.
   \bibitem{Fickle} H. C. Fickle, \emph{Knots, $\Z$-homology $3$-spheres and contractible $4$-manifolds}, 
   Houston J. Math., \textbf{10} (1984), no. 4, pp. 467--493.
  \bibitem{FS} R. Fintushel and R. Stern, \emph{An exotic free involution on $S^4$}, 
   Ann. of Math. (2),\textbf{113} (1981), no. 2, pp. 357--365. 
  \bibitem{Freedman} M. Freedman, \emph{The topology of four-dimensional manifolds},
   J. Differential Geometry, \textbf{17} (1982), no. 3, pp. 357--453.
  \bibitem{GG} D. Gabai, D. Gay, D. Hartman, V. Krushkal and M. Powell, \emph{Pseudo-isotopies of simply connected $4$-manifolds}, 
   Forum Math. Pi, \textbf{14} (2026), e9.
 \bibitem{G-fillability} P. Ghiggini, \emph{Ozsv\'ath-Szab\'o invariants and fillability of contact structures},
   Math. Z., \textbf{253} (2006), no. 1, pp. 159--175.
  \bibitem{Hedden} M. Hedden, \emph{An Ozsv\'ath-Szab\'o Floer homology invariant of knots in a contact manifold}, Adv. Math.,  \textbf{219} (2008) no. 1 pp. 89--117.
  \bibitem{HR} M. Hedden and K. Raoux, \emph{Knot Floer homology and relative adjunction inequalities},
 Selecta Math. (N.S.), \textbf{29} (2023), no. 1, 48 pp.
  \bibitem{LiLiu} T. J. Li and A. K. Liu, \emph{Uniqueness of symplectic canonical class, surface cone and symplectic cone of $4$-manifolds with $b^+=1$},
   J. Differential Geom., \textbf{58} (2001), no. 2, pp. 331--370.  
 \bibitem{MT2} T. Mark and B. Tosun, \emph{Obstructing pseudoconvex embeddings and contractible Stein fillings for Brieskorn spheres},
   Adv. Math., \textbf{335} (2018), pp. 878--895.
  \bibitem{Nemethi0} A. N\'emethi, \emph{On the Ozsv\'ath-Szab\'o invariant of negative definite plumbed $3$-manifolds}, 
  Geom. Topol., \textbf 9 (2005), pp. 991--1042.
 \bibitem{OSz-negative} P. Ozsv\'ath and Z. Szab\'o, \emph{Absolutely graded Floer homologies and intersection forms for four-manifolds with boundary},
  Adv. Math., \textbf{173} (2003), pp. 179--261.  
  \bibitem{OSz-fullpath} P. Ozsv\'ath and Z. Szab\'o, \emph{On the Floer homology of plumbed three-manifolds},
  Geom. Topol., \textbf 7 (2003), no. 1, pp. 185--224.
  \bibitem{OSz-contact} P. Ozsv\'ath and Z. Szab\'o, \emph{Heegaard Floer homology and contact structures},
 Duke Math. J., \textbf{129} (2005), no. 1, pp. 39--61. 
  \bibitem{Olga} O. Plamenevskaya, \emph{Contact structures with distinct Heegaard Floer invariants}, 
   Math. Res. Let., \textbf{11} (2004), pp. 547--561.
  \bibitem{Quinn} F. Quinn, \emph{Isotopy of $4$-manifolds}, 
   J. Differential Geom., \textbf{24} (1986), no. 3, pp. 343--372. 
   \bibitem{Salamon2} D. Salamon, \emph{Removable singularities and a vanishing theorem for Seiberg-Witten invariants},
   Turkish J. Math., \textbf{20} (1996), pp. 61--73.
  \bibitem{Salamon} D. Salamon, \emph{Spin geometry and Seiberg-Witten invariants},
  preprint, 1999. 
 \bibitem{Savk} O. \c{S}avk, \emph{More Brieskorn spheres bounding rational balls},
  Topology Appl., \textbf{286} (2020), 107400, 11 pp. 
\end{thebibliography}
\end{document}